\documentclass[11pt]{amsart}
\usepackage{amssymb,amsthm,amsmath,amstext}
\usepackage{mathrsfs}  
\usepackage{bm}        
\usepackage{mathtools} 
\usepackage{color}
\usepackage[bbgreekl]{mathbbol}
\usepackage{multirow}
\usepackage{enumitem}
\usepackage{cite}

\usepackage[left=3cm,right=3cm]{geometry}

\usepackage{graphicx}

\usepackage[all]{xy}
\usepackage{tikz}
\usepackage{tikz-cd}
\usepackage{mathabx}

\theoremstyle{plain}
\newtheorem{theorem}{Theorem}[section]

\newtheorem{proposition}[theorem]{Proposition}
\newtheorem{lemma}[theorem]{Lemma}
\newtheorem{corollary}[theorem]{Corollary}

\newtheorem{assumption}{Assumption}

\numberwithin{equation}{section}

\theoremstyle{definition}
\newtheorem{definition}[theorem]{Definition}

\newtheorem{conjecture}[theorem]{Conjecture}

\theoremstyle{remark}
\newtheorem{remark}[theorem]{Remark}

\newcommand{\Q}{\mathbb Q}

\DeclareMathOperator{\Aut}{\mathrm{Aut}}

\newcommand{\id}{\mathrm{id}}

\usepackage[backref=page]{hyperref}

\newcommand{\bP}{\mathbb{P}}

\newcommand{\bR}{\mathbb{R}}

\newcommand{\bQ}{\mathbb{Q}}
\newcommand{\bZ}{\mathbb{Z}}
\newcommand{\bF}{\mathbb{F}}
\newcommand{\bC}{\mathbb{C}}

\newcommand{\calC}{\mathcal{C}}

\newcommand{\calK}{\mathcal{K}}
\newcommand{\calH}{\mathcal{H}}

\newcommand{\calO}{\mathcal{O}}
\newcommand{\calL}{\mathcal{L}}

\newcommand{\rank}{\mathrm{rank}}

\DeclareMathOperator{\Burn}{Burn}

\newcommand{\git}{/\kern-0.2em/}

\newcommand{\tH}{\widetilde{\mathrm{H}}}
\newcommand{\AXprim}{A(X)_{prim}}
\newcommand{\Aprim}{A_{prim}}

\newcommand{\Ku}{\mathrm{Ku}}

\begin{document}

\title[Fourier--Mukai partners]{Counting Fourier--Mukai partners of Cubic fourfolds}

\author[B\"ohning]{Christian B\"ohning}\thanks{For the purpose of open access, the authors have applied a Creative Commons Attribution (CC BY) licence to any Author Accepted Manuscript version arising from this submission.}
\address{Christian B\"ohning, Mathematics Institute, University of Warwick\\
Coventry CV4 7AL, England}
\email{C.Boehning@warwick.ac.uk}

\author[von Bothmer]{Hans-Christian Graf von Bothmer}
\address{Hans-Christian Graf von Bothmer, Fachbereich Mathematik der Universit\"at Hamburg\\
Bundesstra\ss e 55\\
20146 Hamburg, Germany}
\email{hans.christian.v.bothmer@uni-hamburg.de}

\author[Marquand]{Lisa Marquand}
\address{Lisa Marquand, Rutgers University\\
Hill Center, Busch Campus\\
110 Frelinghuysen Road\\
Piscataway\\
NJ 8019, USA}
\email{lisa.marquand@rutgers.edu}

\date{\today}


\begin{abstract}
We develop an algorithm to count the number of virtual Fourier--Mukai partners for a given cubic fourfold, with initial input the primitive algebraic lattice and transcendental Hodge structure. Under some mild assumptions, we prove that our virtual count is enough to recover the actual count of Fourier--Mukai partners. We apply our algorithm to examples of cubics with a symplectic automorphism. In particular, we prove that a general cubic with a symplectic involution has 1120 non-trivial Fourier--Mukai partners, each admitting an Eckardt involution, and all birational. As a corollary, we prove that admitting a symplectic automorphism is not a Fourier--Mukai invariant for cubic fourfolds.
\end{abstract}

\maketitle 

\section{Introduction}\label{s:Intro}
The derived category of coherent sheaves of a smooth, complex cubic hypersurface $X\subset \bP^5$ contains a full triangulated subcategory, often called the Kuznetsov component of $X$:
$$\Ku(X):=\langle \calO_X, \calO_X(1), \calO_X(2)\rangle^\perp\subset D^b(X).$$
It is widely conjectured that the Kuznetsov component captures the birational geometry of $X$. Indeed, for some special cubic fourfolds, $\Ku(X)$ is equivalent to the derived category of a $K3$ surface, and Kuznetsov conjectured that this is a necessary and sufficient condition for the cubic fourfold to be rational \cite[Conjecture 1.1]{kuzcubic}. A different, but related, conjecture is the following:

\begin{conjecture}\label{Huy conj}\cite[Conjecture 2.5]{HuybrechtsK3Update}
    Let $X$ and $Y$ be two cubic fourfolds. If $X$ and $Y$ are Fourier--Mukai partners, i.e., if there exists an equivalence $\Ku(X)\simeq \Ku(Y),$ then $X$ and $Y$ are birational.
\end{conjecture}
There are few known examples of pairs $X, Y$ that satisfy the above conjecture: the first example was given in \cite{FanLaiCremona}, who proved that every cubic $X$ containing a Veronese surface had a unique non trivial Fourier--Mukai partner $Y$ (also containing a Veronese), and the pair are indeed birational. Further examples are given in \cite{Brooke:2024aa} - most notably, the authors proved that for a cubic fourfold $X$ containing a non-syzygetic pair of cubic scrolls, there is a non-trivial Fourier--Mukai partner $Y$ (also non-syzygetic) birational to $X$. This partner $Y$ was explicitly constructed from $X$ in \cite{BBM25}.

In order to study this conjecture, one would like a greater source of pairs of Fourier--Mukai partner cubic fourfolds. For a cubic fourfold $X$ whose algebraic lattice $A(X)=H^{2,2}(X)\cap H^4(X,\bZ)$ has rank 2, i.e., a very general special cubic fourfold, counting formulas for the number of Fourier--Mukai partners were established by Fan--Lai in \cite{FanLai23}, generalising special cases of formulas in \cite{FanLaiCremona}, \cite{pert21}, and building on work by Oguiso, who studied the analogous question for very general polarised $K3$ surfaces \cite{Og02}. Formulas for the number of Fourier--Mukai partners of any $K3$ surfaces were established in \cite{HLO04}, in terms of certain double cosets obtained from orthogonal groups of the Néron-Severi lattice. A Fourier--Mukai equivalence between $K3$ surfaces induces a Hodge isometry between the transcendental cohomology $T$ - their formula is based on counting certain possible primitive embeddings of this lattice $T$ into the $K3$ lattice, using Nikulin's theory \cite{nikulin}. 
Similarly, a Fourier--Mukai equivalence for cubic fourfolds also induces a Hodge isometry between the transcendental cohomology, but now the situation becomes more complex. 
First, the lattice $\Lambda\cong H^4(X,\bZ)_{prim}$ is no longer unimodular. 
Second, certain embeddings of $T(X)$ into $\Lambda$ do not correspond to smooth cubic fourfolds: in other words,  $T(X)^\perp$ may not correspond to the algebraic primitive cohomology of a smooth cubic fourfold.
Hence a fully closed form formula is out of reach.

In this paper, we establish an algorithm to calculate the number of Fourier--Mukai partners for a cubic fourfold, under some mild assumptions. 
We assume at the outset one has a cubic fourfold with specified algebraic primitive cohomology $A(X)_{prim}:=H^{2,2}(X)\cap H^4(X,\bZ)_{prim}$ and transcendental cohomology $T(X):= A(X)_{prim}^\perp.$ 
A Fourier--Mukai equivalence between cubic fourfolds induces a Hodge isometry between the Addington-Thomas lattices \cite{AT14}, restricting to one of the transcendental cohomology, but
it is possible that the partner cubic has non-isomorphic primitive algebraic lattice.
We start by identifying possible lattices that occur as orthogonal complements of $T(X)$ for some primitive embedding $T(X)\hookrightarrow \Lambda.$ For each such lattice $K$, we develop a procedure for counting so-called \textbf{virtual Fourier--Mukai partners} with algebraic lattice $K$; these are isomorphism classes of overlattices $K\oplus T\subset L$ with induced Hodge structure.
Our count is in terms of the orthogonal group of potential primitive algebraic lattices $K$, as one would expect from the counting formula of \cite{HLO04}.

However, the number of virtual Fourier--Mukai partners is sometimes much larger than the actual count of Fourier--Mukai partners.
Indeed, for $L$ as above, there exists a cubic fourfold $Y$ with a Hodge isomorphism $H^4(Y,\bZ)_{prim}\cong L$ if and only if there does not exist $v\in K$ with $v^2=2$, or $v^2=6$ and divisibility in $L$ equal to 3 \cite[Thm. 1.1]{laz10}. 
Unfortunately, for a fixed $K$ some overlattices $L$ will satisfy this condition, whereas others will not (see Section \ref{s:SymplecticOrder3} for explicit example). 
Fortunately, under mild assumptions we can correct the virtual count to give an actual count, by extracting the appropriate $L$.

We then apply our algorithm in examples, illustrating that our assumptions are indeed mild (i.e. satisfied in many situations of geometric interest), and how to extract the final count of Fourier--Mukai partners from the virtual count. The main source of geometrically interesting cubic fourfolds comes from those admitting symplectic automorphisms - we focus on those of order 2 and 3. Recall that a cubic fourfold $X$ with a symplectic involution $\phi$ contains 120 pairs of cubic scrolls and has $A(X)_{prim}\cong E_8(2)$ \cite[Theorem 1.1]{Mar23}. We show:

\begin{theorem}\label{t:invintro}
    Let $X$ be a general cubic fourfold with a symplectic involution $\phi.$ Then $X$ has $1120$ nontrivial Fourier--Mukai partners $Y$, each with $A(Y)_{prim}=E_6(2)\oplus A_2(2).$ Further, each $Y$ is birational to $X$.
\end{theorem}
Note that the lattice $E_8$ contains exactly 1120 $A_2$ root subsystems (see Lemma \ref{l:A2SubsysttemsE8}). Each subsystem corresponds to a different pair of nonsyzygetic cubic scrolls contained in $X$ - each Fourier--Mukai partner is then obtained via the construction of \cite{BBM25}.

We also apply our algorithm to the case of cubic fourfolds with a symplectic automorphism of order $3$ of type $\phi_3^6$ (see Theorem \ref{t:Z3} for notation). We prove:

\begin{theorem}\label{t:order3intro}
    Let $X$ be a general cubic fourfold with a symplectic automorphism of order $3$ of type $\phi_3^6$. Then $X$ has $623$ nontrivial Fourier--Mukai partners $Y$. Of these, 350 admit a symplectic automorphism of the same type, and the remaining 273 do not admit any automorphism. Further, each such $Y$ is birational to $X$.
\end{theorem}
Indeed, there are 273 Fourier--Mukai partners that have different primitive algebraic lattice to that of $X$. However, in contrast to the case of involutions, we do not have a geometric explanation for the existence of the partners. Both $X$ and each Fourier--Mukai partner $Y$ can be shown to be rational, hence Conjecture \ref{Huy conj} holds in both cases.

It is worth highlighting that both examples prove the following corollary:
\begin{corollary}
    The existence of a symplectic automorphism is not preserved under Fourier--Mukai partnership of cubic fourfolds.
\end{corollary}
In the case of symplectic involutions, this is direct contrast to $K3$ surfaces: the existence of a Nikulin involution is preserved under Fourier--Mukai partnership by \cite[Proposition 9.7]{HTinvK3}.

\subsection*{Outline} In Section \ref{s:Prelim} we recall the lattice theoretic results we will need, along with basic results about Fourier--Mukai partners of cubic fourfolds. In Section \ref{sec: the virtual count} we develop the notion of virtual Fourier--Mukai partners, and develop our algorithm for counting the number. In Section \ref{sec: actual count}, we prove various propositions that ensure our virtual count gives the correct count for Hodge theoretic Fourier--Mukai partners, and explain how to extract the correct count otherwise. In Section \ref{sec:inv}, we apply our algorithm to the case of symplectic involutions and prove Theorem \ref{t:invintro}. Finally, in Section \ref{s:SymplecticOrder3} we study cubic fourfolds with symplectic automorphism of order 3 and prove Theorem \ref{t:order3intro}.

\subsection*{Acknowledgements} We would like to thank Stevell Muller for valuable help and advice on primitive embeddings and working with OSCAR. 
Many of our computations were verified in both Magma and OSCAR,  \cite{Magma}, \cite{OSCAR}, \cite{OSCAR-book}. We would like to thank the anonymous referee for useful comments and suggestions.
L.~Marquand was supported by NSF grant DMS-2503390. She was also supported in part by an AMS-Simons travel grant, and would like to thank The University of Warwick where this article was initiated.

For the purpose of open access, the first author has applied a Creative Commons Attribution (CC-BY) licence to any Author Accepted Manuscript version arising from this submission.

\section{Preliminaries}\label{s:Prelim}
We recall the theory that we will need in order to develop our counting algorithm. In Section \ref{SubSec:LatticeTheory} we recall the basics of lattice theory and Nikulin's theory of overlattices. In Section \ref{Subsec:FiniteOrth} we recall results on automorphism groups of orthogonal vector spaces over $\bF_3.$ In Section \ref{SubSec:Cubic Fourfolds} we recall Hodge theory of cubic fourfolds, and in Section \ref{subsec: FM partners} the basics of Fourier--Mukai partners of a cubic fourfold.
\subsection{Lattice theory}\label{SubSec:LatticeTheory}

\subsubsection{Finite symmetric bilinear and quadratic forms}\label{SubSubSec:FiniteQuadratic}
We recall a few facts about finite symmetric bilinear forms and finite quadratic forms following \cite[\S 1, $2^{\circ}$]{nikulin}. Recall that if $D$ is a finite abelian group, a finite symmetric bilinear form is simply a symmetric bilinear form $b \colon D \times D \to \bQ/\bZ$, and a finite quadratic form is a map $q\colon D \to \bQ/2\bZ$ such that 
\begin{eqnarray*}
q(na) = n^2 q(a) \text{ for all } n\in \bZ \text{ and } a\in D, \\
q(a+a') - q(a) -q(a') \equiv 2 b(a,a') \, (\mathrm{mod}\, 2\bZ) , \quad a, a'\in D
\end{eqnarray*}
where $b\colon D\times D \to \bQ/\bZ$ is a finite symmetric bilinear form. Clearly, $q=\oplus_p q_p$ and $b=\oplus_pb_p$ where $q_p$ and $b_p$ are the restrictions of $q$ and $b$ to the $p$-component $D_p$ of the finite group $D$. 

Note that, writing $\bQ^{(p)}$ for the set of rational numbers that can be written in lowest terms with a denominator a power of $p$, we have a natural isomorphism
\[
i\colon \bQ /\bZ \simeq \bigoplus_p \bQ^{(p)}/\bZ .
\]
Indeed, the inverse $i^{-1}$ is induced by the natural inclusions $\bQ^{(p)} \hookrightarrow \bQ$. Applying the multiplication by $2$ map to both sides of the isomorphism $i$ yields an isomorphism
\[
i_1\colon \bQ /2\bZ \simeq \Bigl( \bigoplus_{p\neq 2}  2\bQ^{(p)}/2\bZ \Bigr) \oplus \bQ^{(2)}/2\bZ . 
\]
We then get, \cite[Prop. 1.2.3]{nikulin}, that $q_p$ maps $D_p$ to $2\bQ^{(p)}/2\bZ$ for $p\neq 2$, and $q_2$ maps $D_2$ to $\bQ^{(2)}/2\bZ$, whereas $b_p$ maps $D_p \times D_p$ to $\bQ^{(p)} / \bZ$ for all $p$ using the above isomorphisms. Moreover, if $p\neq 2$, then $b_p$ (contrary to what happens in general) \emph{determines} $q_p$ by the rule
\[
q_p(a) = m_2 \left( 2^{-1} b_p(a,a) \right)
\]
where $a\in D_p$ and $m_2 \colon \bQ^{(p)} /\bZ \to 2\bQ^{(p)} /2\bZ$ is the multiplication by $2$ isomorphism. Here $2^{-1} b_p(a,a)$ has to be representable by an element in $\bQ^{(p)}$, well defined up to addition of an element in $\bZ$, so to get such a representative, we need to pick a representative of $b_p(a,a)$ as a rational number with \emph{even} numerator and denominator a power of $p$ (in lowest terms), which is always possible because $p$ is odd (and the class of $2^{-1} b_p(a,a)$ in $\bQ^{(p)} /\bZ$ is then well-defined using this procedure). 

\subsubsection{Gluing theory and overlattices}\label{SubSubSec:Gluing}

The count of Fourier--Mukai partners will eventually be reduced to a count of certain even overlattices of a given even lattice, and this will be accomplished using the theory in this paragraph. 

For the following compare \cite{nikulin} or \cite{DolgachevQuadratic}.
If $S$ is an even (non-degenerate) sublattice of an even (non-degenerate) lattice $S'$ of the same rank, then the group $H_{S, S'}=S'/S$ is a subgroup of the discriminant group $D(S)$ on which the quadratic form $q_S$ vanishes identically - an \emph{isotropic subgroup}. Conversely, given an isotropic subgroup $H$ of $D(S)$, its preimage $S'$ in $D(S)=S^*/S$ is a subgroup of $S^*$ on which the $\bQ$-valued symmetric bilinear form on $S^*$ induces the structure of an even non-degenerate lattice containing $S$ as a sublattice of the same rank. This gives a bijection
\[
S' \leftrightarrow H_{S',S}
\]
between the set of non-degenerate lattices containing $S$ as a sublattice of finite index and the set of isotropic subgroups of $D(S)$. Moreover, under this correspondence, $D(S')$ is identified with 
\[
H_{S, S'}^{\perp}/H_{S,S'}
\]
with $q_{S'}$ being the restriction of $q_S$ to $D(S')$ \cite[Prop. 1.4.1]{nikulin}, \cite[Lemma 1.4.4]{DolgachevQuadratic}. 

\medskip

In particular, the above reasoning applies when when we consider a primitive embedding of an even lattice $S$ into an even lattice $M$, with orthogonal complement $K$. Then $M$ is an overlattice of $S\oplus K$. Therefore, it is classified by the preceding by the isotropic subgroup 
\[
H_{M}:= M/(S\oplus K) \subset D(S)\oplus D(K).
\]
The fact that the embeddings of $S$ and $K$ into $M$ are primitive gives one piece of additional information: the projections $p_S\colon H_{M} \to D(S)$ and $p_K\colon H_{M} \to D(K)$ are embeddings (this is equivalent to the embeddings of $S$ and $K$ into $M$ being primitive). Denoting by $H_{M, S}$ and $H_{M, K}$ the images under the the two projections to $S$ and $K$, we see 
\[
\gamma_{S,K}^M = p_K\circ p_S^{-1}\colon H_{M,S} \to H_{K,S}
\]
is an isomorphism under which $q_K\circ \gamma_{S,K}^M =-q_S$ because $H_M$ is an isotropic subgroup. This discussion can be summarized in the following theorem:

\begin{theorem}\label{theorem:glueing}\cite[1.5.1]{nikulin}
    Let $S, K$ be even lattices. Then even   overlattices $M$ of $S\oplus K$ are in bijection with the following data:
    \begin{itemize}
        \item a subgroup $H_{M,S}\leq D(S)$,
        \item a subgroup $H_{M, K}\leq D(K)$,
        \item an anti-isometry $\alpha: H_{M, S}\rightarrow H_{M,K}$.
    \end{itemize}
    Further, $(D(M), q_M)\cong (\Gamma_\alpha^\perp/\Gamma_\alpha, (q_S\oplus q_K)|_{\Gamma_\alpha^\perp/\Gamma_\alpha})$, where $\Gamma_\alpha=M/(S\oplus K)\subset D(S)\oplus D(K)$ is the graph of $\alpha.$
\end{theorem}

\subsubsection{Uniqueness and classification of primitive embeddings}\label{SubSubSec:Embeddings}

We will need results about primitive embeddings of one even lattice into another in two situations.

\begin{theorem}\cite[1.14.4]{nikulin}\label{theorem: unique embedding into AT lattice}
Let $T$ be an even lattice of signature $(t_+, t_{-})$ and let $H$ be an even unimodular lattice of signature $(h_+, h_{-})$. Then there exists a \emph{unique} primitive embedding of $T$ into $H$ provided the following hold. 

\begin{enumerate}
\item $h_+-t_+ > 0$ and $h_- - t_- >0$. 
\item 
$h_+ + h_- - (t_++t_-) \ge l(D(T)_p) + 2$ for all primes $p\neq 2$, where $l(D(T)_p)$ denotes the minimum number of generators of the abelian group $D(T)_p$, the $p$-part of $D(T)$. 
\item 
If $h_+ + h_- - (t_++t_-) = l(D(T)_2)$, then $q_T=u_+^{(2)}(2)\oplus q'$ or $q_T= v_+^{(2)}(2)\oplus q'$ where $u_+^{(2)}(2)$ and $v_+^{(2)}(2)$ are as in \cite[Prop. 1.8.1]{nikulin} and $q'$ some other quadratic form. 
\end{enumerate}
\end{theorem}

\medskip

In the following we will, for an even lattice $S$, fix its invariants $(s_+, s_-, q_S)$, the signature $(s_+, s_-)$ and the finite quadratic form $q_S\colon D(S) \to \bQ/2\bZ$ (so in particular, we also fix the abelian group $D(S)$). By \cite[Corollary 1.9.4]{nikulin} fixing the invariants $(s_+, s_-, q_S)$ is \emph{equivalent} to fixing the genus of $S$. In other words, all lattices $S'$ with the same invariants are precisely the ones such that $S'\otimes\bZ_p \simeq S\otimes \bZ_p$ for all primes $p$ and $S' \otimes \bR \simeq S \otimes \bR$. Such $S'$ are finite in number and it is in principle possible to enumerate all isometry classes of these $S'$.  

\begin{theorem}\cite{nikulin}[1.15.1]\label{theorem: prim embedding}
Let $S$ be an even lattice with invariants $(s_+, s_-, q_S)$, and let $M$ be an even lattice with invariants $(m_+, m_-, q_M)$. We wish to understand all primitive embeddings of $S$ into $M$. Such an embedding is determined by the following data:
\[
(H_S, H_M, \gamma; K, \gamma_K)
\]
where: 
\begin{enumerate}
\item 
$H_S \subset D(S)$ and $H_M \subset D(M)$ are subgroups and $\gamma \colon (H_S, q_S|_{H_S})\rightarrow (H_M, q_M|_{H_M})$ is an isometric isomorphism of finite groups.
\item
For one of the choices in a), set $$\delta:=(q_S\oplus -q_M)|_{\Gamma_\gamma^\perp/\Gamma_\gamma}$$
where $\Gamma_\gamma$ is the graph of $\gamma$ in $D(S)\oplus D(M).$ Let $K$ be an even lattice with invariants
\[
(m_+-s_+, m_--s_-, -\delta ). 
\]
This amounts to enumerating all lattices $K$ in that genus. 
\item 
$\gamma_K:(D(K),q_K)\rightarrow (\Gamma_\gamma^\perp/\Gamma_\gamma, -\delta)$ is an isomorphism.
\end{enumerate}
\end{theorem}

Note that if $M$ happens to be unimodular, then $D(M)$ is trivial, so $H_M$ and $H_S$ are, too, and $\Gamma_{\gamma}^{\perp}/\Gamma_{\gamma}$ will be all of $D(S)$. Thus in this case, any even lattice $K$ in the genus specified by $(m_+-s_+, m_--s_-, -\delta)$ where $-\delta =-q_S\colon D(S) \to \bQ/2\bZ$ can be the orthogonal complement of a primitive embedding of $S$ into $M$ and all of these occur. 

In cases, when $D(M)$ is not trivial, classifying all possible such complements to an embedding of $S$ into $M$ amounts to double enumeration: first, we have to enumerate the possible $\gamma's$. For example, if $M=\Lambda$ with discriminant group $D(M)=\bZ/3$, then $H_M$ might be trivial (whence we have to enumerate all $K$ in the genus of $(m_+-s_+, m_--s_-, -\delta)$ where $-\delta =-q_S \oplus q_M\colon D(S)\oplus \bZ/3 \to \bQ/2\bZ$); or $H_M=\bZ/3$, whence the task consists in locating a subgroup $\bZ/3$ of $D(S)$ isometric to $D(M)$, choosing an isometry $\gamma$, and computing $\delta$ on $\Gamma_{\gamma}^{\perp}/\Gamma_{\gamma}$. 

\subsection{Finite orthogonal geometries and their automorphism groups}\label{Subsec:FiniteOrth}

For our lattice theoretic counts below we will need several results about orthogonal vector spaces over $\bF_3$ and their automorphism groups. We collect the relevant theory in this subsection. General references are \cite[Chapter VII]{Dickson}, \cite[Chapter VII.3]{AdemMilgram}, \cite{Taylor}, \cite{Dieudonne},  or \cite[Chapter 9]{Grove}. 

Let $K$ be a finite field of order $p^\nu$ where $p$ is an odd prime. Then 
\[
K^*/(K^*)^2 \simeq \bZ/2
\]
and there are $\frac{1}{2}(p^{\nu}-1)$ squares and the same number of non-squares in $K^*$. Consider a finite-dimensional vector space $V$ over $K$ with a quadratic form $q\colon V \to K$, or equivalently, a symmetric bilinear form $b \colon V \times V \to K$. For most of what follows we will assume that $V$ is \emph{regular}, i.e. that $q$ and $b$ are nondegenerate in the sense that $b\colon V \to V^*$ is an isomorphism. 

The \emph{determinant} of a regular quadratic space $V$ is the element $\det(V) \in K^*/(K^*)^2$ defined by the determinant of the Gram matrix of $b$ with respect to some basis. 

We use the standard notation
\[
\langle a_1, \dots , a_r \rangle
\]
for the diagonal quadratic form $a_1x_1^2 + \dots + a_rx_r^2$. If $1$ and $\epsilon$ represent the two elements of $K^*/(K^*)^2$, then $\langle 1, 1 \rangle \simeq \langle \epsilon, \epsilon \rangle$: this follows because by the pigeon-hole principle, every element of $K$ is the sum of two squares, and if $\alpha^2+\beta^2 =\epsilon$, the base change $x' = \alpha x + \beta y$, $y' = \beta x - \alpha y$ transforms $\langle 1, 1 \rangle$ into $\langle \epsilon, \epsilon \rangle$. Together with the standard diagonalization procedure for quadratic forms over fields of characteristic not $2$ this yields

\begin{proposition}\label{p:IsomClassesFiniteQuadratic}
For every dimension $n$ there are exactly two isometry classes of regular quadratic spaces $V$ over $K$:
\[
\langle 1, 1, \dots , 1 \rangle 
\]
and 
\[
\langle 1, 1, \dots , \epsilon \rangle 
\]
($n$ entries in both brackets). Two regular quadratic spaces over $K$ are isometric if and only if they have the same dimension and determinant. 
\end{proposition}

Note that if $n$ is odd, then 
\[
\langle 1, 1, \dots , \epsilon \rangle \simeq \langle \epsilon, \dots , \epsilon \rangle
\]
and the forms
\[
x_1^2 + \dots + x_n^2 \quad \mathrm{and}\quad  \epsilon (x_1^2 + \dots + x_n^2)
\]
have isometric automorphism groups. Thus if $n$ is odd, there is up to isomorphism only one orthogonal group of a regular quadratic space over $K$ whereas if $n$ is even, there are two such orthogonal groups. 

More precisely, following \cite[Chapter VII.3]{AdemMilgram}, we denote by $H_{2n}$ the sum of $n$ hyperbolic forms
\[
H_{2n}:=\left\langle \begin{pmatrix} 0 &1 \\ 1 & 0  \end{pmatrix} \right\rangle \perp \dots \perp \left\langle \begin{pmatrix} 0 &1 \\ 1 & 0  \end{pmatrix} \right\rangle 
\]
and the associated orthogonal group by $\operatorname{O}^+(2n, K)$. 
Moreover, consider the form
\[
H_{2n}':= \langle H_{2n-2} \rangle \perp \left\langle \begin{pmatrix} 1 & 0 \\ 0 & -\epsilon  \end{pmatrix} \right\rangle
\]
for $\epsilon$ as above a non-square in $K$, and the associated orthogonal group $\operatorname{O}^-(2n, K)$.

\begin{remark}\label{r:SpecialCaseF3}
For $K=\bF_3$, we can take $\epsilon = -1$. Then $H_{2n}$ has determinant $(-1)^n$ whereas $H_{2n}'$ has determinant $(-1)^{n-1} \cdot (-\epsilon)= (-1)^n \epsilon = (-1)^{n+1}$. In general, $H_{2n}$ and $H_{2n}'$ can be distinguished by computing their determinants. 
\end{remark}

Then we have \cite[Chapter VII.3, Thm. 3.5]{AdemMilgram}

\begin{theorem}\label{tOrderOrthogonal}
The orders of the orthogonal groups of a regular orthogonal vector space of dimension $2n+1$ resp. $2n$ over a finite field $K$ with $k=p^{\nu}$, $p\neq 2$,   are as follows:
\begin{align*}
\left| \operatorname{O}(2n+1, K) \right| &= 2 k^{n^2} \prod_{i=1}^n \left( k^{2i} - 1 \right), \\
\left| \operatorname{O}^+(2n, K) \right| &= 2 k^{n(n-1)} (k^n - 1) \prod_{i=1}^{n-1} \left( k^{2i} - 1 \right), \\
\left| \operatorname{O}^-(2n, K) \right| &= 2 k^{n(n-1)} (k^n + 1) \prod_{i=1}^{n-1} \left( k^{2i} - 1 \right).
\end{align*}
\end{theorem}

\begin{remark}\label{rOrthogonalGroupDegenrate}
    If a vector space over $K$ equipped with a possibly degenerate quadratic form $q$ is decomposed into $V=R\oplus V'$ where $R$ is the radical of $q$ and $V'$ is a complement on which $q$ is nondegenerate, then 
    \[
\operatorname{O}(V) \cong  \operatorname{Hom}(V', R) \rtimes \left(\operatorname{GL}(R)  \times  \operatorname{O}(V')\right),
\]
since any isometry has to preserve $R$ and then acts on the quotient $V/R \simeq V'$. Hence in this case
\[
|\operatorname{O}(V)| = |\operatorname{GL}(R)| \cdot \underbrace{k^{d n}}_{\text{$\dim \operatorname{Hom}(V', R)$}} \cdot |\operatorname{O}(V')|,
\]
where $d = \dim(R)$, $n = \dim(V')$, 
\(|\operatorname{GL}(R)| = \prod_{i=0}^{d-1} (k^d - k^i)\) and $|\operatorname{O}(V')|$ is computed as in Theorem \ref{tOrderOrthogonal}.
\end{remark}

For the applications below, we will be interested in counting injective isometries from one regular quadratic space over $K$ into another. 

\begin{proposition}\label{p:CountingIsometries}
Let $V$ and $W$ be regular quadratic spaces over $K$ with $\dim W =\dim V +1$. Then the number of injective isometries of $V$ into $W$ is given by 
\[
\frac{\left| \mathrm{O}(W) \right| }{2}. 
\]
\end{proposition}

\begin{proof}
If $V\simeq \langle 1, 1, \dots , 1 \rangle$, then there is an injective isometry of $V$ into $W$, and in fact, $W \simeq V \perp \langle 1 \rangle$ or $W \simeq V \perp \langle \epsilon \rangle$. If $V\simeq \langle 1, 1, \dots , \epsilon \rangle$ ($n$ entries), then $W\simeq \langle 1, 1, \dots , 1 \rangle$ or $\simeq \langle 1, 1, \dots , \epsilon \rangle$ ($n+1$ entries), but since $\langle 1, 1 \rangle \simeq \langle \epsilon, \epsilon \rangle$, in the first case $W \simeq V \perp \langle \epsilon \rangle$ whereas in the second case $W \simeq V \perp \langle 1 \rangle$. In particular, there exists an injective isometry $V \to W$ also in this case, and $W$ in all cases splits into a subspace isometric to $V$ and an orthogonal complement on which the quadratic form is nondegenerate. By Witt's extension theorem, all injective embeddings $V \to W$ form one orbit under $\mathrm{O}(W)$. The stabilizer of a fixed injective embedding is the orthogonal group of the complement of the copy of $V$ inside $W$ specified by the embedding, so this is $\bZ/2$. This implies the formula given in the statement. 
\end{proof}

We will also need one result for orthogonal vector spaces over $\bF_2$ below.

Following \cite[Chapter 12]{Grove}, we call a finite-dimensional vector space $V$ over a field $K$ of characteristic $2$ equipped with a quadratic form $q$ \emph{nondefective} if the radical of $V$ is zero (the more usual term nondegenerate being reserved for a different concept in this context in many sources). 

Also, according to \cite[Thm. 12.9]{Grove}, if $\dim V=n=2m$ is even and $K$ is perfect, $V$ nondefective, then there is a basis $\{v_i\}$ of $V$ such that 
\[
q(\sum_i a_i v_i) = \sum_{i=1}^m a_ia_{i+m}
\]
which is called \emph{type 2}, or
\[
q(\sum_i a_i v_i) = \sum_{i=1}^{m-1} a_ia_{m-1+i} + a_{2m-1}^2 + a_{2m-1}a_{2m} + b a_{2m}^2 
\]
with $x^2 + x + b$ irreducible in $K[x]$ (which is called \emph{type 3}). 

\begin{theorem}\label{t:OrderOrthF2}
Suppose that $K$ has $k=2^r$ elements and $V$ is a nondefective quadratic space over $K$ of dimension $n=2m$. If $V$ is of type $2$ then
\[
|\mathrm{O}(V)|= 2k^{\frac{n(n-2)}{4}}\left( k^{\frac{n}{2}} -1 \right) \prod_{i=1}^{\frac{n-2}{2}} \left( k^{2i} -1\right),
\]
and if $V$ is of type $3$ then
\[
|\mathrm{O}(V)|= 2k^{\frac{n(n-2)}{4}}\left( k^{\frac{n}{2}} +1 \right) \prod_{i=1}^{\frac{n-2}{2}} \left( k^{2i} -1\right).
\]
\end{theorem}

\begin{proof}
This is \cite[Thm. 14.48]{Grove}. 
\end{proof}

\subsection{Cubic fourfolds}\label{SubSec:Cubic Fourfolds}

The middle cohomology $H^4(X,\bZ)$ of a smooth cubic fourfold $X$ is an odd unimodular lattice of signature $(21,2)$. We let $\eta\in H^4(X,\bZ)$ be the class of the square of the hyperplane class. The primitive cohomology $H^4(X,\bZ)_{prim}:=\langle \eta\rangle^\perp$ is an even lattice of signature $(20,2)$, but is no longer unimodular. In fact, $H^4(X,\bZ)_{prim}\cong \Lambda,$ where $$\Lambda:= U^{\oplus 2}\oplus E_8^{\oplus 2} \oplus A_2.$$
In particular, the discriminant group $D(\Lambda):=\Lambda^*/\Lambda \cong \bZ/3\bZ,$ with quadratic form $q$ that satisfies $q(\xi)=\frac{2}{3} \mod 2\bZ$ where $\xi$ is a generator.

We define the following lattices:
\begin{align*}
    A(X)&:=H^{2,2}(X)\cap H^4(X,\bZ)\\
    A(X)_{prim}&:=A(X)\cap \langle \eta\rangle^\perp.
\end{align*}
We call $A(X)$ (resp. $A(X)_{prim}$) the algebraic lattice (resp. primitive algebraic lattice) of $X$.
We define the transcendetal lattice $T(X):=A(X)_{prim}^\perp\subset H^4(X,\bZ)_{prim},$ equipped with the induced Hodge structure. Equivalently, $T(X)=A(X)^\perp\subset H^4(X,\bZ).$ 

For a given lattice $T$ admitting a primitive embedding into $\Lambda,$ we will often need to identify when the lattice $A:=T^\perp\subset \Lambda$ is isometric to $A(X)_{prim}$ for some cubic fourfold. We have the following criterion, which follows from the description of the image of the period map for cubic fourfolds. Recall that for a vector $v\in L$ the $div_L(v)=\max\{\lambda \in \bZ \mid v\cdot L\in \lambda\bZ\}$. 

\begin{theorem}\cite[Theorem 1.1]{laz10} \label{thm: no roots}
    Let $X$ be a smooth cubic fourfold. Then there does not exist either:
    \begin{enumerate}
        \item $v\in A(X)_{prim}$ with $v^2=2,$ or
        \item $v\in A(X)_{prim}$ with $v^2=6$ and $div_L(v) = 3$.
    \end{enumerate}
    Further, if $K\hookrightarrow \Lambda$ is a primitive embedding of a positive definite even lattice $K$ satisfying the above two conditions, then there exists a smooth cubic fourfold $X$ with $A(X)_{prim}=K.$
\end{theorem}

In \cite{AT14}, Addington and Thomas introduced a Hodge structure associated with the $K3$ category $\Ku(X),$ defined on the lattice $K_{top}(\Ku(X))$ equipped with the Euler pairing.  We briefly recall the construction.
 There is an embedding (given by the Mukai vector) of 
 $v: K_{top}(\Ku(X))\hookrightarrow H^*(X,\bQ)$, where $E\mapsto ch(E)\sqrt{td(X)}$. This defines a weight 2 Hodge structure on $K_{top}(\Ku(X))$, denoted by $\tH(\Ku(X),\bZ)$, defined by
 $\tH^{2,0}(\Ku(X), \bC)=v^{-1}(H^{3,1}(X, \bC))$. We often call the $\tH(\Ku(X),\bZ)$ the Addington-Thomas lattice of $X$.

As an abstract lattice, $\tH(\Ku(X),\bZ)\cong \widetilde{\Lambda},$ where
$$\widetilde{\Lambda}:=U^{\oplus 4} \oplus E_8^{\oplus 2},$$
an even unimodular lattice of signature $(20,4)$. 

\subsection{Fourier--Mukai partners}\label{subsec: FM partners}
Let $X\subset \bP^5$ be a smooth cubic fourfold. There exists a semiorthogonal decomposition of $D^b(X)$:
$$D^b(X)=\langle \Ku(X),\calO_X, \calO_X(1), \calO_X(2)\rangle,$$ and we call $\Ku(X)$ the Kuznetsov component of $X$.
We say that two smooth cubic fourfolds $X, Y$ are \textbf{Fourier--Mukai partners} if there exists an equivalence of categories $$\Ku(X)\simeq \Ku(Y).$$ 

The focus of this article is to develop an algorithmic procedure to count the number of Fourier--Mukai partners for a given cubic fourfold - by \cite[Theorem 1.1]{Huy17} this is a finite number. Recall the following:

\begin{proposition}\cite[Proposition 3.4]{Huy17}
    Let $X,Y$ be cubic fourfolds. Then any Fourier--Mukai equivalence $\Ku(X)\simeq \Ku(Y)$ induces a Hodge isometry 
    $\tH(\Ku(X),\bZ)\cong \tH(\Ku(Y),\bZ).$
\end{proposition}

Below we will often abbreviate $\tH(\Ku(X),\bZ )$ to $\tH(X)$ for ease of notation. 

On the other hand, we have the following standard conjecture \cite[Conjecture 6.1]{HuybrechtsK3Update}.

\begin{conjecture}\label{conj:HodgeEqualsDerived}
Two cubic fourfolds $X$ and $Y$ are Fourier--Mukai partners if and only if there exists an orientation-preserving Hodge isometry $\tH(X) \simeq \tH(Y)$. 
\end{conjecture}

Note that the negative directions of $\tH(\Ku(X),\bZ)$ (and similarly for $Y$) come with a natural orientation given by the real and imaginary parts of 
$\tH^{2,0}(\Ku(X), \bC)$ and the oriented basis $\lambda_1, \lambda_2$ of $-A_2\subset \tH^{1,1}(\Ku(X), \bZ)$ where $\lambda_1, \lambda_2$ are the projections of $\mathcal{O}_l(1)$ and $\mathcal{O}_l(2)$, for a line $l\subset X$, into $K_{top}(\Ku(X))$, see \cite[\S 3.1]{Huy17}.
Here by orientation-preserving Hodge isometry we mean a Hodge isometry that preserves these natural orientations.

Note that by \cite[Lemma 2.3]{Huy17} there is always an orientation-reversing Hodge isometry of $\tH(X)$ in the case of a smooth cubic fourfold $X$. Hence Conjecture \ref{conj:HodgeEqualsDerived} will imply

\begin{conjecture}\label{conj:HodgeGivesDerived}
Two cubic fourfolds $X$ and $Y$ are Fourier--Mukai partners if and only if $\tH(X)$ and $\tH(Y)$ are Hodge isometric. 
\end{conjecture}

We can establish the preceding conjecture for various concrete examples of cubics. However, because a general proof seems to be lacking, and our count will naturally yield the numbers of cubics with isometric Addington-Thomas Hodge structures, it seems reasonable to make the following definition.

\begin{definition}\label{d:HFMPartners}
We call smooth cubic fourfolds $X$ and $Y$ \emph{Hodge-theoretic Fourier--Mukai partners} if $\tH(X)$ and $\tH(Y)$ are Hodge isometric.
\end{definition}

One general result that is useful is

\begin{theorem}\label{t:TwistedK3s}
Let $X$ and $Y$ be cubic fourfolds whose Kuznetsov components are equivalent to derived categories of twisted K3 surfaces. Then $X$ and $Y$ are Fourier--Mukai partners if and only if they are Hodge-theoretic Fourier--Mukai partners. 
\end{theorem}

\begin{proof}
This is due to Mukai, Orlov, Huybrechts-Stellari, compare \cite[Chapter 7, Theorem 3.17]{Huybrechtscubicsbook}.
\end{proof}

\section{The virtual count}\label{sec: the virtual count}

Let $X$ be a smooth cubic fourfold.  At the outset we assume we are given its algebraic lattice $A:=A(X)$ together with the class $\eta=h^2\in A$, and its transcendental lattice $T:=T(X)$. Recall that $T\cong A^\perp\subset H^4(X,\bZ)$. Note that $T$ is an even lattice.

From this input datum we compute the following:
\begin{enumerate}
    \item We compute $\Aprim:= \AXprim \coloneqq \eta^\perp\subset A.$ This is an even positive definite lattice.
    \item We compute the discriminant group $(D(\Aprim), q_{\Aprim})$ directly.
    \item We compute the discriminant group $(D(T),q_{T})$ directly.
\end{enumerate}

We make the following simplifying assumptions:

\begin{assumption}\label{assumption 1}
    We have $\rank (A)<21,$ and $X$ is a very general cubic with the given $A(X)=A$ and $T(X)=T$.
\end{assumption}
The above assumption gives that $\mathrm{End}_{Hdg}(T(X)\otimes \Q)=\bQ$ by \cite[Theorem 1.5.1]{zarhin}. It follows that the only Hodge isometries of the transcendental cohomology $T(X)$ for such a cubic are $\pm \id_{T(X)}$.

\begin{assumption}\label{assumption 3}
    The transcendental lattice $T$ satisfies the conditions of Theorem \ref{theorem: unique embedding into AT lattice} where $H:=\tilde{H}(\Ku(X),\bZ)$ is the Addington-Thomas lattice. 
\end{assumption}
This assumption will be used to prove  that any isometry $T(X)\cong T(Y)$ for two cubic fourfolds $X$ and $Y$ extends to an isometry of the corresponding Addington-Thomas lattices. 

\medskip

\begin{center}
\fbox{\textbf{From now on, we always  assume that Assumptions 1 and 2 hold.}}
\end{center}

\medskip

Our ultimate aim is to compute the number of Fourier--Mukai partners of $X$. We will start by counting so called Virtual Fourier--Mukai partners:

\begin{definition}
    Let $X$ be as above, and let $T\cong T(X)$. 
    Let $K$ be an orthogonal complement of $T$ for some primitive embedding $T\hookrightarrow \Lambda$ such that there does not exist $v\in K$ with $v^2=2.$

    A \textbf{virtual Fourier--Mukai Partner} of $X$ with algebraic lattice $K$ is an integral weight 4, rank 22 Hodge structure $L$ satisfying:
    \begin{enumerate}
         \item\label{a} $K\oplus T\subset L\subset K^*\oplus T^*$ and the Hodge structure on $L$ is induced by declaring all classes in $K$ to be algebraic and declaring the Hodge structure on $T$ to be the given one,
        \item $K, T$ are both saturated in $L$
        \item\label{c} $L$ has discriminant 3.
    \end{enumerate}
    We denote the set of such \emph{overlattices} $L$ satisfying a), b), c) (and ignoring Hodge structures) by $\calL(K,T)$.
    
    We denote the set of isomorphism classes of virtual Fourier--Mukai partners of $X$ with algebraic lattice $K$ by $VFM(X,K).$ 

    A virtual Fourier--Mukai partner of $X$ is a Hodge structure $L$ such that $L\in VFM(X,K)$ for some $K$ (up to isomorphism of Hodge structures). We denote this set $VFM(X)$.
\end{definition}
The conditions ensure that $L$ is isomorphic as a lattice to $\Lambda$. Then if $K$ in addition contains no vectors $v\in K$ with $v^2=6$ and $div_L(v)=3,$ by Theorem \ref{thm: no roots} then there exists a unique cubic fourfold $Y$ with $H^4(Y,\bZ)_{prim}\cong L$ as Hodge structures. Thus, by applying Assumption \ref{assumption 3}, the set $VFM(X,K)$ contains the set of Hodge structures of Hodge-theoretic Fourier--Mukai partners of $X$.  

Note that there is a natural surjective map $\Pi: \calL(K,T)\rightarrow VFM(X,K)$,  where $\Pi(L)=\Pi(L')$ if there exists an isomorphism of Hodge structures $L\overset{\sim}\rightarrow L'.$

\begin{definition}\label{dGroupActing}
Let
\[
G_{K, T}:= O(K) \times \{\pm \mathrm{id}_T \}
\]
be the product group of the group of lattice isometries $O(K)$ of $K$ and the group $\bZ/2$ consisting of $\pm \mathrm{id}_T$ acting on $T$. Note that $G_{K, T}$ acts on $K\oplus T$ and on $K^*\oplus T^*$.
\end{definition}

\begin{lemma}
    Let $f:L\overset{\sim}\rightarrow L'$ be an isomorphism of Hodge structures with $L,L'\in \calL(K,T)$. Then $f$ determines an element in $\bar{f}\in G_{K,T}$. Conversely, an element $g\in G_{K,T}$ determines such an isomorphism of Hodge structures, where $L'=g(L)$.
\end{lemma}
\begin{proof}
    Restricting the isomorphism $f$ to $K\oplus T$ gives the element $\bar{f}$. The converse is clear.
\end{proof}
We immediately obtain:
\begin{corollary}\label{c:Orbits}
    The number $|VFM(X, K)|$ of isomorphism classes of virtual Fourier--Mukai partners with algebraic lattice $K$ is equal to the number of $G_{K,T}$-orbits in $\calL(K,T).$
\end{corollary}

To count the total number of Fourier--Mukai partners of $X$, we now count and determine the possible lattices $K$.

\begin{definition}\label{d:List}
    Let $\mathscr{K}=\{K_i\}_{i=1}^n$ be a list of all lattices (up to isometry) that occur as $T^\perp\subset \Lambda$ for some primitive embedding $T\hookrightarrow \Lambda$ and have the property that there is no vector $v\in K_i$ with $v^2=2$.
\end{definition}

It follows that $VFM(X)=\bigsqcup_{K\in \mathscr{K}} VFM(X,K)$.

\begin{remark}
    Note that not every lattice $K\in\mathscr{K}$ can occur as the primitive algebraic lattice of a cubic fourfold. Indeed, $K$ occurs if there exists an embedding $K\hookrightarrow \Lambda$ such that every class $v\in K$ with $v^2=6$ has divisibility 1 in $\Lambda$. Unfortunately, it is possible for one $K$ to admit several embeddings into $\Lambda,$ some with this property and some without. This is why our current count is `virtual'. We explain how to extract the actual count in Section \ref{sec: actual count}.
\end{remark}

\begin{definition}\label{d:isotropicSubgroupsDiscriminantGroup}
 Let $\calH_{K, T}$ be the set of isotropic subgroups
\[
H \subset D(K)\oplus D(T)
\]
such that the projections $p_K\colon H \to D(K)$ and $p_T \colon H \to D(T)$ are embeddings and 
\[
H^{\perp}/H \simeq \bZ/3.
\]
\end{definition}

Note that by the theory recalled in Subsection \ref{SubSubSec:Gluing} and more specifically Theorem \ref{theorem:glueing}, elements $L$ in $\calL(K, T)$ correspond bijectively to 
elements of $\calH_{K, T}$ via $L\mapsto H_L:= L/(K\oplus T)$. 
Further, we see that $G_{K, T}$ acts naturally on $\calH_{K, T}$. There is also a natural homomorphism
\[
\rho \colon G_{K, T} \to O(D(K)) \times \{ \pm \mathrm{id}_{D(T)} \}.
\]
We denote $O(D(K)) \times \{ \pm \mathrm{id}_{D(T)}\}$ by $\overline{G}_{K, T}.$ We can summarize this discussion in the following proposition:

\begin{proposition}\label{pCount2}
The number $|VFM(X,K)|$ of virtual Fourier--Mukai partners (up to isomorphism), with primitive algebraic lattice isomorphic to $K$, of a cubic fourfold $X$ is given by the number of $G_{K,T}$-orbits in $\calH_{K, T}$. A lower bound for this number is given by the number of $\overline{G}_{K, T}$-orbits in $\calH_{K, T}$. 

The number of $\overline{G}_{K, T}$-orbits in $\calH_{K, T}$ equals the number $|VFM(X,K) |$ if the homomorphism $\rho \colon G_{K, T} \to \overline{G}_{K, T}$ is surjective. 
\end{proposition}

\begin{proof}
The number of $\overline{G}_{K, T}$-orbits in $\calH_{K, T}$ provides a lower bound because every $G_{K, T}$-orbit is contained in a $\overline{G}_{K, T}$-orbit. The last assertion is then clear. 
\end{proof}

We now establish general results for the number of $\overline{G}_{K, T}$-orbits in $\calH_{K, T}$ under the following simplifying assumption:

\begin{assumption}\label{a:CountViaOrthGroups}
The discriminant groups $D(K)$ and $D(T)$ have the following forms:
\begin{enumerate}
\item 
\textbf{Case 1.}
\begin{align*}
    D(K)= & (\bZ/3)^r \oplus \bigoplus_{j=1}^N A_{p_j},\\
    D(T) = & (\bZ/3)^{r+1} \oplus \bigoplus_{j=1}^N A_{p_j}
\end{align*}
where $r\in \mathbb{N}$ and $p_j\neq 3$ are primes, $A_{p_j}$ is a finite group of order a power of $p_j$.
\item 
\textbf{Case 2.}
\begin{align*}
    D(K)= & (\bZ/3)^r \oplus \bigoplus_{j=1}^N A_{p_j},\\
    D(T) = & (\bZ/3)^{r-1} \oplus \bigoplus_{j=1}^N A_{p_j}
\end{align*}
where $r\in \mathbb{N}$ and $p_j\neq 3$ are primes, $A_{p_j}$ is a finite group of order a power of $p_j$.
\end{enumerate}
Moreover, the $3$-parts of the discriminant quadratic forms
\[
q_{K, 3}  \colon D(K)_3 \to \bQ/(2\bZ), \quad q_{T, 3}  \colon D(T)_3 \to \bQ/(2\bZ)
\]
are non-degenerate. 
\end{assumption}

Using the theory in Subsection \ref{SubSubSec:FiniteQuadratic}, we can view $q_{K, 3}$ and $q_{T, 3}$ as quadratic forms with values in $\bF_3$ on the $\bF_3$-vector spaces $D(K)_3$ and $D(T)_3$. 
We now explain how one can compute the number of 
$\overline{G}_{K,T}$-orbits in $\calH_{K, T}$ using the theory recalled in Subsection \ref{Subsec:FiniteOrth}. 

We start with Case 1 of Assumption \ref{a:CountViaOrthGroups}. Note that the quadratic form $q_K$ on $D(K)=  (\bZ/3)^r \oplus \bigoplus_{j=1}^N A_{p_j}$ splits as a direct sum
\[
q_K = q_{K, 3} \oplus \bigoplus_{j=1}^N q_{K, p_j}
\]
and on $D(T) =  (\bZ/3)^{r+1} \oplus \bigoplus_{j=1}^N A_{p_j}$ the quadratic form $q_T$ then has the form
\[
q_T = q_{T, 3} \oplus \bigoplus_{j=1}^N (-q_{K, p_j}). 
\]
Accordingly, $\overline{G}_{K, T}$ splits as
\[
\overline{G}_{K, T} = \Bigl( \mathrm{O}(D(K)_3, q_{K, 3}) \times \bigtimes_{j=1}^N \mathrm{O}(D(K)_{p_j}, q_{K, p_j}) \Bigr) \times \{ \pm \mathrm{id}_{D(T)}\}.
\]
The isotropic subgroups $H_L\subset D(K) \times D(T)$ in the set $\calH_{K, T}$ in this case are precisely given by graphs $\Gamma_{\alpha}$ of injective anti-isometries $\alpha \colon D(K) \to D(T)$. Since we want to count the number of orbits of such graphs for the action of $\overline{G}_{K, T}$, the above discussion yields the following:

\begin{proposition}\label{p:NumberGraphsCase1}
In Case 1 of Assumption \ref{a:CountViaOrthGroups}, the number of $\overline{G}_{K, T}$-orbits in $\calH_{K, T}$ is given by the number of $\mathrm{O}(V)$-orbits of graphs of isometries (not anti-isometries!) 
\[
\iota \colon V \to W
\]
where $V$ is the regular orthogonal vector space $(\bF_3^r, q_{K, 3})$ and $W$ is the regular orthogonal vector space $(\bF_3^{r+1}, -q_{T,3})$ (note the minus sign here in front of $q_{T,3}$). 
\end{proposition}

\begin{proof}
An injective anti-isometry as above $\alpha:D(K)\rightarrow D(T)$ restricts to an anti-isometric isomorphism $\alpha_p:D(K)_p\rightarrow D(T)_p$ for $p\neq 3$. Any two choices differ by an element of $O(D(K)_p),$ so form a single orbit - thus the orbit of $\alpha$ is determined entirely by the $3$-primary orbit. 

Writing $V=(D(K)_3,q_{K,3})$ and $W=(D(T)_3, -q_{T,3})$ and setting $\iota=\alpha_3$ gives the injective isometry as described in the Proposition. 

The only thing left to notice is that $\overline{G}_{K, T}$ contains the element $(-\mathrm{id}_K, -\mathrm{id}_T)$ which acts trivially on the graph $\Gamma_{\alpha}$ of every anti-isometry. Hence two such graphs are in the same $\overline{G}_{K, T}$-orbit if and only if they are in the same $\mathrm{O}(D(K))$-orbit. 
\end{proof}

\begin{proposition}\label{p:CountCase1}
In Case 1 of Assumption \ref{a:CountViaOrthGroups}, the number of $\overline{G}_{K, T}$-orbits in $\calH_{K, T}$ is given by $\frac{|\mathrm{O}(W)|}{2 |\mathrm{O}(V)|}.$
\end{proposition}

\begin{proof}
Combine Proposition \ref{p:NumberGraphsCase1} with Proposition \ref{p:CountingIsometries}. 
\end{proof}

In Case 2 of Assumption \ref{a:CountViaOrthGroups}, almost everything said above goes through, but the isotropic subgroups $H_L\subset D(K) \times D(T)$ in the set $\calH_{K, T}$ are now given by graphs $\Gamma_{\alpha}$ of injective anti-isometries $\alpha \colon D(T) \to D(K)$ (source and target interchanged compared to Case 1). Thus we get

\begin{proposition}\label{p:NumberGraphsCase2}
In Case 2 of Assumption \ref{a:CountViaOrthGroups}, the number of $\overline{G}_{K, T}$-orbits in $\calH_{K, T}$ is given by the number of $\mathrm{O}(W)$-orbits of graphs of isometries
\[
\iota \colon V \to W
\]
where $V$ is the regular orthogonal vector space $(\bF_3^{r-1}, -q_{T, 3})$ and $W$ is the regular orthogonal vector space $(\bF_3^{r}, q_{K,3})$. 
\end{proposition}

\begin{proposition}\label{p:CountCase2}
In Case 2 of Assumption \ref{a:CountViaOrthGroups}, the number of $\overline{G}_{K, T}$-orbits in $\calH_{K, T}$ is equal to $1$. 
\end{proposition}
\begin{proof}
    By Proposition \ref{p:CountingIsometries}, the group $\mathrm{O}(W)$ acts transitively on the set of injective isometries $V\to W$. The result therefore follows immediately from Proposition \ref{p:NumberGraphsCase2}.
\end{proof}

\section{From virtual to actual count}\label{sec: actual count}
From the previous section, we are able to count the number of virtual Fourier--Mukai Partners $|VFM(X)|$ of a cubic fourfold $X$ under Assumptions \ref{assumption 1}, \ref{assumption 3}, and \ref{a:CountViaOrthGroups}. 
Each element $[L]\in VFM(X)$ is abstractly isomorphic to the lattice $\Lambda$, with a Hodge structure such that $T(L)\cong T(X),$ and by Assumption \ref{assumption 3} this Hodge isometry extends to one of $H\cong \tilde{H}(\Ku(X),\bZ).$ 

It follows that provided there exists a cubic fourfold $Y$ such that $H^4(Y,\bZ)_{prim}\cong L$ as Hodge structures, then $Y$ is an honest Hodge theoretic Fourier--Mukai partner of $X$. Unfortunately, this may not happen for all $K\in \calK$, cf. Remark \ref{r:DifferenceVirtualActual}.  
Here, we explain how to extract the true count of Hodge theoretic Fourier--Mukai Partners from the virtual count.

Recall that $VFM(X)=\bigsqcup_{K\in \mathscr{K}} VFM(X,K)$. Let $HFM(X,K)$ denote the set of Hodge theoretic Fourier--Mukai partners of $X$ with algebraic primitive lattice isomorphic to $K$.

\begin{proposition}\label{pPrimitiveAlgebraicFMPartner}
Let $Y\in HFM(X).$ Then $Y\in HFM(X,K)$ for some $K\in \calK$. Further, there is an injective map $HFM(X,K)\rightarrow VFM(X,K).$
\end{proposition}
\begin{proof}
    Let $Y\in HFM(X).$ Then there exists a Hodge isometry between Addington-Thomas lattices $\widetilde{H}(X)\cong \widetilde{H}(Y),$ restricting to a Hodge isometry $T(X)\cong T(Y).$ 
    Then there is a primitive embedding $$T\cong T(Y)\hookrightarrow H^4_{prim}(Y,\bZ)\cong \Lambda,$$ and $A_{prim}(Y)\cong (T^\perp)\subset \Lambda.$ By \ref{thm: no roots}, the lattice $A_{prim}(Y)$ does not contain a vector $v$ with $v^2=2$.
    It follows that $A_{prim}(Y)$ is isomorphic to a lattice $K$ appearing in the list $\mathscr{K},$ and $Y\in HFM(X,K).$

    Further, setting $L:=H^4(Y,\bZ)_{prim},$ we see that $[L]\in VFM(X,K);$ this defines the desired map. The global Torelli theorem \cite{voisintorelli} asserts that this is injective.
\end{proof}

We wish to find conditions on $K$ such that the injection $HFM(X,K)\rightarrow VFM(X,K)$ is surjective, guaranteeing our virtual count is an actual count. One immediately sees the following:

\begin{lemma}\label{lem: no square 6}
    Let $K\in\mathscr{K}$ such that there are no vectors $v\in K$ with $v^2=6$
    and $div_K(v)=3$ or $div_K(v)=6$. Then $|HFM(X,K)|=|VFM(X,K)|$.
\end{lemma}
    
\begin{proof}
       Suppose that $K\in\mathscr{K}$, satisfying the additional assumption that there are no $v\in K$ with $v^2=6$ and $div_K(v)=3$ or $div_K(v)=6$. By assumption, each $[L]\in VFM(X,K)$ is an isomorphism class of overlattices $K\oplus T\subset L$ with $L\cong \Lambda$, both $K$ and $T$ primitively embedded into $L$, and equipped with a unique Hodge structure on inducing the given Hodge structure on $T$ and such that all classes in $K$ are algebraic. 
       There are no $v\in K$ with $v^2=6$ and $div_L(v)=3$; indeed, if there was, then $div_K(v)=3$ or $6$ since $D(K)_3=(\bZ/3)^r$.
       Thus by \ref{thm: no roots}, there exists a cubic fourfold $Y$ such that $H^4(Y,\bZ)_{prim}\cong L$ as Hodge structures. 
       By construction, this cubic $Y$ has transcendental lattice Hodge isometric to $T(X)$. By Assumption \ref{assumption 3}, we obtain $\widetilde{H}(\Ku(X),\bZ) \simeq \widetilde{H}(\Ku(Y), \bZ)$. Thus $Y\in HFM(X,K)$. This defines an injective inverse $VFM(X,K)\rightarrow HFM(X,K)$, and the claim follows.
\end{proof}

\begin{remark}
    Suppose $K\in \mathscr{K}$ has a vector $v\in K$ with $v^2=6$ and $div_K(v)=3$. It is still possible that there exists some cubic fourfold $Y$ with $\Aprim(Y)\cong K$. Indeed, if there exists $L\in VFM(X,K)$ such that $div_L(v)=1$ in $L$ for all $v\in K$ with $v^2=6$, then such a $Y$ exists by the Torelli theorem.

    The problem is such a $K$ may have several primitive embeddings into $\Lambda$, which may or may not satisfy the divisibility condition. This is why our initial count is `virtual' - to pass from virtual to actual, one needs to distinguish these cases, compare Remark \ref{r:DifferenceVirtualActual}. 
\end{remark}

\begin{proposition}\label{prop: case 1}
    Let $K\in \calK$ be such that we are in Case 1 in Assumption \ref{a:CountViaOrthGroups}. Then $|HFM(X,K)|=|VFM(X,K)|$
\end{proposition}
\begin{proof}
    Let $[L]\in VFM(X,K),$ in particular an (isomorphism class of) overlattice $K\oplus T\subset L$, corresponding to group $H_L:=L/(K\oplus T)\subset D(K)\oplus D(T)$. By Assumption \ref{a:CountViaOrthGroups}, it follows that $H_L\cong D(K).$ By Lemma \ref{lem: no square 6}, it remains to assume that there exists $v\in K$ with square 6 and $div_K(v)=3$ or 6. Then $[\frac{v}{div(v)}]\in D(K),$ and there must exist $w\in T$ such that $\frac{v+w}{3}\in L$. Then $v\cdot \frac{v+w}{3}= 2$ and so $div_L(v)\neq 3.$

    By the same argument as in Lemma \ref{lem: no square 6}, there exists a cubic fourfold $Y$ with $H^4(Y,\bZ)_{prim}\cong L$ as Hodge structures. Again, Assumption \ref{assumption 3} ensures that $Y\in HFM(X,K)$.
\end{proof}

For the $K\in \mathscr{K}$ with discriminant group $D(K)$ satisfying Case 2 in Assumption \ref{a:CountViaOrthGroups}, and admitting $v\in K$ with $v^2=6, div_K(v)=3$ or $6$ the situation is more subtle.
Let $[L]\in VFM(X,K)$ corresponding to $H_L=L/(K\oplus T).$
It follows that $H_L\cong D(T)$, and the isomorphism class of $L$ corresponds to a choice of isomorphism $D(K)\cong D(T)(-1)\oplus \Gamma $, where $\Gamma\cong \bZ/3\bZ.$ Let $v\in K$ with $v^2=6$ and $div_K(v)=3$ or $6$. 
If the image $[\frac{v}{3}]\in D(K)$ for every such $v$ is contained in $D(T)(-1)$, then $div_L(v)=1$ and $[L]$ corresponds to some $Y\in HFM(X,K)$ as before.
Thus let $VFM(X,K)'\subset VFM(X,K)$ denote the subset containing all $[L]\in VFM(X,K)$ satisfying this condition. Again, applying the same argument of Lemma \ref{lem: no square 6}, we immediately obtain the following:

\begin{lemma}\label{lem:VFM'}
    Let $K\in \calK$ such that the discriminant group $D(K)$ is of the form Case 2 in Assumption \ref{a:CountViaOrthGroups}. Then $|HFM(X,K)|=|VFM(X,K)'|.$
\end{lemma}

\begin{proof}
It remains to observe that every Hodge-theoretic Fourier--Mukai partner arises, via Torelli, from an element in $VFM(X,K)'$: indeed, if we had a vector $v$, $v^2=6$, in $K$ such that $\frac{v}{3}$ surjects onto the discriminant group of $L$, then $div_{L}(v)=3$. So elements in $VFM(X, K) \backslash VFM(X,K)'$ do not give rise to Hodge structures coming from the middle primitive cohomology of a cubic fourfold. 
\end{proof}

\section{Example 1: cubic fourfolds with symplectic involution}\label{sec:inv}

We will now illustrate in examples how the theory of the preceding sections yields actual counts of Fourier--Mukai partners for various geometrically and algebraically interesting classes of cubic fourfolds. A good source of such comes from cubics with automorphisms. In particular, the algebraic, algebraic primitive lattices and transcendental lattices are known for cubic fourfolds with prime order symplectic automorphism (see \cite{laza2019automorphisms} and \cite{BGM25}). One could apply this theory to every such example - here we showcase two examples, which highlight the technically difficult points of the count we have described.

In this section we focus on cubic fourfolds with symplectic involution. We prove:

\begin{theorem}
    Let $X$ be a general cubic fourfold with symplectic involution. Then $X$ has exactly 1120 non-trivial Fourier--Mukai partners $Y$, with $A(Y)_{prim}\cong E_6(2)\oplus A_2(Y).$
\end{theorem}

In particular, none of the Fourier--Mukai partners admit a symplectic involution. Curiously they all admit a antisymplectic involution. Thus we obtain the following corollary:

\begin{corollary}
    Admitting a symplectic automorphism is not preserved under Fourier--Mukai partnership for cubic fourfolds.
\end{corollary}

\subsection{Counting Hodge theoretic Fourier--Mukai partners}\label{s:SymplecticInvolution}

 We recall the main classification theorem for cubic fourfolds with involution \cite[Thm.1.1]{Mar23}. 
\begin{theorem}\label{thm: inv}
	Let $X$ be a  general cubic fourfold with $\phi_i$ an involution of $X$ fixing a linear subspace of $\bP^5$ of codimension $i$. Then either:
	\begin{enumerate}
		\item $i=1$, $\phi_1$ is anti-symplectic and $A(X)_{prim}\cong E_6(2)$, $T(X)\cong U^2\oplus D_4^3.$ The algebraic lattice is spanned by classes of planes contained in $X$;
		\item $i=2$, $\phi_2$ is symplectic and $A(X)_{prim}\cong E_8(2)$, $T(X)\cong A_2\oplus U^2\oplus E_8(2).$ The algebraic lattice is spanned by classes of cubic scrolls contained in $X$;
		\item $i=3$, $\phi_3$ is anti-symplectic and $$A(X)_{prim}\cong M,\,\, T(X)\cong U\oplus A_1\oplus A_1(-1)\oplus E_8(2).$$ The algebraic lattice contains an index 2 sublattice spanned by classes of planes contained in $X$.
	\end{enumerate}
Here $M$ is the unique rank 10 even lattice obtained as an index 2 overlattice of $D_9(2)\oplus\langle 24\rangle$.
\end{theorem}

\begin{remark}
Note that cubics with an involution of type $\phi_1$ are precisely the ones with an Eckardt point (see \cite{LPZ}).
\end{remark}

Here we will focus on case (b) in Theorem \ref{thm: inv}: cubic fourfolds with a symplectic involution. Let $X$ be a general cubic with symplectic involution. Thus:
\begin{enumerate}
    \item $A(X)_{prim}\cong E_8(2)$ with $D(A(X)_{prim})=(\bZ/2)^8$;
    \item $T:=T(X)\cong A_2\oplus U^2\oplus E_8(2)$, with $D(T)=\bZ/3\times (\bZ/2)^8$.
\end{enumerate}

\medskip

Note that Assumption \ref{assumption 1} clearly holds in this case. Assumption \ref{assumption 3} holds as well since, with the notation of Theorem \ref{theorem: unique embedding into AT lattice}, we have that $H=\widetilde{\Lambda}$ has signature $(h_+, h_-)=(20, 4)$, $T$ has signature $(t_+, t_-)=(12,2)$. Condition (a) of Theorem \ref{theorem: unique embedding into AT lattice} thus holds, for (b) we get $h_++h_- -(t_++t_-)= 10$ and thus (c) is vacuous and (b) is true because $10 \ge l(D(T)_3)+2=3.$ 

\medskip

We first determine the list $\mathscr{K}$, consisting of all lattices $K$ up to isometry that arise as orthogonal complements of $T$ for some primitive embedding $T\hookrightarrow \Lambda.$

\begin{lemma}
    We have $\mathscr{K}=\{E_8(2), E_6(2)\oplus A_2(2)\}$.
\end{lemma}
\begin{proof}
    We use Theorem \ref{theorem: prim embedding} for $S:=T$ and $\Lambda:=M$. We first enumerate genera for possible $T^\perp$ - there are two possibilities based on the groups $H_S, H_M$. 

    The first corresponds to the genus of $A(X)_{prim}\cong E_8(2).$ This corresponds to $H_T\cong H_\Lambda\cong \bZ/3$ in the notation of Theorem \ref{theorem: prim embedding}. Since $E_8(2)$ is unique in its genus, every primitive embedding of $T$ into $\Lambda$ with $H_T\cong \bZ/3\bZ$ has $T^\perp\cong E_8(2).$

    The second possibility corresponds to when $H_T=H_\Lambda =(1).$ Such an embedding exists if there exists an even lattice $K$ with signature $(8,0)$ and $D(K)= (\bZ/2)^8\times (\bZ/3)^2,$ with quadratic form 
    
\[
q=-q_{T(X)} \oplus q_\Lambda\colon D(T(X))\oplus \bZ/3 \to \bQ/2\bZ .
\]
  The $2$-part $q_2$ is the same as the $2$-part of the quadratic form on $D(E_8(2))$, and this can be computed \cite{MMSO_Files} to have Conway-Sloane genus symbol at $2$ given by $2^8$.  We have to check the symbol at $3$. By construction, it is represented by the form with matrix $\mathrm{diag}(+1, -1)$ on $(\bZ/3)^2$ (from the factors of $A_2$ in $T(X)$ and $\Lambda$), and its determinant is $-1$, a quadratic nonresidue, so the symbol at $3$ is $3^{-2}$ as given above. Thus the Conway-Sloane genus symbol is given as: 
  \[\mathrm{at}\: 2: \quad  2^8 \quad \mathrm{at}\: 3: \quad 1^{-6}\, 3^{-2}. \]

    A representative for this genus is exactly $E_6(2)\oplus A_2(2)$. One can check via OSCAR \cite{MMSO_Files} that this is the only isometry class in this genus that does not contain any vectors of square 2. This completes the list $\calK.$  
\end{proof}
We now compute $|HFM(X,K)|$ for each $K\in \mathscr{K}.$ 

\begin{proposition}
    We have $HFM(X,E_8(2))=\{X\}.$
\end{proposition}
\begin{proof}
    Note first that by either Lemma \ref{lem: no square 6} or Proposition \ref{prop: case 1}, the set $VFM(X,E_8(2))$ is isomorphic to the set $HFM(X,E_8(2))$. Set $K:=E_8(2)$ in what follows. By Proposition \ref{pCount2}, a lower bound for $|VFM(X,K)|$ is the number of $\overline{G}_{K,T}$-orbits in $\calH_{K,T}.$ By Proposition \ref{p:CountCase1}, this is equal to 1.

    To ensure that the number of $\overline{G}_{K, T}$-orbits in $\calH_{K, T}$ is equal to the number of Hodge-theoretic Fourier--Mukai partners with $K=E_8(2)$, we can again use Proposition \ref{pCount2} if we show that the homomorphism $\rho \colon G_{K, T} \to \overline{G}_{K, T}$ is surjective. Equivalently, we want 
\[
\mathrm{O}(E_8(2))) \to \mathrm{O}(D(E_8(2)))
\]
to be surjective. 

A Magma calculation \cite{MMSO_Files}  shows that the order of the image of the preceding homomorphism is 
\[
348364800. 
\]
One needs to check that this equals the order of the corresponding orthogonal group of $\bF_2^8 \simeq D(E_8(2))$. By Theorem \ref{t:OrderOrthF2}, we observe that this number is precisely the order of $\mathrm{O}(V)$ for a nondefective quadratic space $V$ of dimension $8$ and type $3$ over $\bF_2$. Thus it remains to prove that $D(E_8(2))$ with its discriminant quadratic form is a space of that type. It is enough to show that $V$ is nondefective, i.e. the radical is zero. Indeed, then 348364800 has to divide the order of $\mathrm{O}(V)$, but it only divides that order for a type 3 space $V$ of dimension 8, not a type 2 space. It remains to compute the bilinear form 
\[
b(x,y) = q(x+y) - q(x) -q(y)
\]
on $D(E_8(2))$ and check it has trivial radical. This follows from a direct computation with Magma \cite{MMSO_Files}. Thus $|VFM(X,E_8(2))|=1$; clearly $X\in VFM(X,E_8(2)).$
\end{proof}

\begin{proposition}\label{prop: non trivial inv  partners}
    We have $|HFM(X,E_6(2)\oplus A_2(2))|=1120$.
\end{proposition}
\begin{proof}

    Let $K:=E_6(2)\oplus A_2(2)$. We see that the set $VFM(X,K)$ is isomorphic to $HFM(X,K)$ by Lemma \ref{lem: no square 6}. 
    A lower bound for $|VFM(X,K)|$ is again the number of $\overline{G}_{K,T}$-orbits in $\calH_{K,T}.$ By Proposition \ref{p:CountCase2}, it is equal to $1$. 

    Again, we have to check if the map $\rho_K \colon \mathrm{O}(K) \to \mathrm{O}(D(K))$ is surjective or not. In this case, it turns out to not be surjective: 

    Using Magma \cite{MMSO_Files}, one finds the following:
\begin{enumerate}
\item
The order of the image of $\rho_K$ is equal to 
\[
1244160.
\]
\item 
The image of 
\[
\rho_{A_2(2)}\colon \mathrm{O}(A_2(2)) \to \mathrm{O}(D(A_2(2)))
\]
has order $12$, so this homomorphism is surjective.
\item 
The image of 
\[
\rho_{E_6(2)}\colon \mathrm{O}(E_6 (2)) \to \mathrm{O}(D(E_6 (2))) 
\]
has order $103680$. Recall that 
\[
D(E_6 (2)) = (\bZ/2)^6 \times \bZ/3
\]
and the discriminant forms $q_{E_6(2),2}$ and $q_{E_6(2),3}$ can be checked to be non-defective resp. non-degenerate. The orthogonal group of $\bZ/3$ thus has order $2$. By Theorem \ref{t:OrderOrthF2}, the order of the orthogonal group of a six-dimensional $V\simeq (\bZ/2)^6$ of type $3$ is $51840$ and $2\cdot 51840=103680$. This proves both that $((\bZ/2)^6, q_{E_6(2),2})$ must be of type $3$, not type $2$, for divisibility reasons, and that $\rho_{E_6(2)}$ is surjective. 
\item 
The image of $\rho_K$ is precisely the product of the images of $\rho_{A_2(2)}$ and $\rho_{E_6(2)}$. Indeed, the former contains the latter and has the same order
\[
12\cdot 103680 = 1244160.
\]
\item 
We already know that the order of the orthogonal group of the $(\mathbb{Z}/2)^8$-part of $D(K)$ is equal to $348364800$ since it agrees with the $(\mathbb{Z}/2)^8$-part of $D(E_8(2))$, which we examined above. Looking back at how the gluing construction works, we see that the $(\bZ/3)^2$-part of $D(K)$ has form $\mathrm{diag}(+1,-1)$ and the orthogonal of that part $(\bZ/3)^2$ is thus equal to $4$. Thus 
\[
|\mathrm{O}(D(K))| = 348364800 \cdot 4 = 1393459200.
\]
Thus
\[
\frac{|\mathrm{O}(D(K))| }{|\mathrm{im}(\rho_K)|} = 1120. 
\] 
\end{enumerate}

In summary, in this case we need to distinguish the notions of $\overline{G}_{K, T}$-orbits in $\calH_{K, T}$ from that of $G_{K, T}$-orbits in $\calH_{K, T}$, and need to count the latter to obtain the number of (a priori Hodge-theoretic) Fourier--Mukai partners with primitive algebraic lattice isomorphic to $K$. By the above calculation, this number is equal to $1120$. Indeed, this number is equal to the number of $\{ \pm 1\} \times \mathrm{O}(K)$-orbits of totally isotropic subgroups of $D(T) \times D(K)$ that are graphs of injective anti-isometries $\alpha \colon D(T) \to D(K)$. We recall that
\[
D(T) \simeq \bZ/3 \times (\bZ/2)^8, \quad D(K) \simeq (\bZ/3)^2 \times (\bZ/2)^8 .
\]
Note that $\{- 1\} \times \{ - 1\} \subset \{ \pm 1\} \times \mathrm{O}(K)$ stabilises each such graph, therefore we can also count $\mathrm{O}(K)$-orbits of such graphs and get the correct number. 
By Witt's theorem, there are $\mathrm{O}(D(K))/2$ many such graphs of anti-isometries. The group $\mathrm{O}(K)$ acts on the set of such graphs, and its image in $\mathrm{O}(D(K))$ contains the subgroup isomorphic to $\bZ/2$ in $\mathrm{O}(D(K))$ that stabilises each graph because $\mathrm{O}(K) \to \mathrm{O}(D(K))$ is surjective onto the $3$-part of $\mathrm{O}(D(K))$. Therefore the number of (Hodge-theoretic) Fourier--Mukai partners with primitive algebraic lattice isomorphic to $K$ is indeed
\[
\frac{|\mathrm{O}(D(K))|/2 }{|\mathrm{im}(\rho_K)|/2}= \frac{|\mathrm{O}(D(K))| }{|\mathrm{im}(\rho_K)|} = 1120. 
\] 
\end{proof}

\subsection{Counting actual Fourier--Mukai partners}
We will prove that all of the $1120$ non-trivial Hodge-theoretic Fourier--Mukai partners of $X$ are actual (derived) Fourier--Mukai partners. As a by-product, we will give a geometric construction of all $1120$ partners with primitive algebraic lattice $E_6(2)\oplus A_2(2)$ in Proposition \ref{prop: non trivial inv  partners} using Gale duality of cubic fourfolds with a nonsyzygetic labelling, as in the article \cite{BBM25}.

\begin{lemma}\label{l:A2SubsysttemsE8}
There are exactly $1120$ root subsystems of type $A_2$ of the root system $E_8$.
\end{lemma}

\begin{proof}
The root lattice $E_8$ (here we work with the positive definite variant) has $240$ roots, vectors $v$ with $v^2=2$. Having chosen $v$, there are $56$ other roots $w$, such that $v\cdot w = -1$. 
This gives $240\cdot 56= 13440$ ordered pairs of roots $(v,w)$ such that the intersection matrix of the sublattice spanned by $v,w$ with respect to the ordered basis $(v,w)$ is 
\[
\begin{pmatrix}
2 & -1\\
-1 & 2 
\end{pmatrix}. 
\]
Now fixing a root subsystem of type $A_2$, there are $12$ ordered pairs $(\alpha , \beta )$ of roots in that subsystem such that the matrix with respect to the basis $(\alpha , \beta)$ has the above form - indeed, we can pick any of the six roots of $A_2$ and then one of the two roots having inner product $-1$ with the first root. Thus the number of $A_2$-subsystems of $E_8$ is 
\[
\frac{13440}{12} = 1120. 
\]
\end{proof}

To construct the $1120$ Fourier--Mukai partners from Proposition \ref{prop: non trivial inv  partners} explicitly, we first recall a few facts about the geometry of (very general) cubic fourfolds $X$ with a symplectic involution of type $\phi_2$ following \cite{Mar23} and \cite{Sturmfels18}. 

The involution $\phi_2$ fixes a linear subspace  $\Pi \simeq \bP^3$ and a complementary line $l\subset X$, hence also the cubic surface $S=X\cap \Pi$. Projection from $l$ exhibits $\mathrm{Bl}_l(X)$ as a conic bundle over $\Pi$ with discriminant locus the union of the cubic surface $S$ and a quadric cone $Q$ in $\Pi$. The intersection $C=S\cap Q \subset \Pi$ is a genus four curve in $\bP^3$ parametrizing fibers of the conic bundle that are double lines. The curve $C$ is a space sextic and a so-called uniquely trigonal genus $4$ curve, see \cite[Section 2]{Sturmfels18}. Let 
\[
\varphi\colon \Sigma \to Q \subset \bP^3
\]
be the induced double cover of $Q$. $\Sigma$ is a degree $1$ del Pezzo surface and the $2:1$ map is induced by $-2K_{\Sigma}$. 

There are $240$ exceptional curves $E_k$ on $\Sigma$. 
We have $\mathrm{Pic}(\Sigma ) \simeq (-K_{\Sigma})\oplus E_8(-1)$. 
The curves $E_k$ correspond in a one-to-one way  to the roots of the $E_8(-1)$ root system via $E_k \mapsto E_k +K_{\Sigma}$. The images of these exceptional curves in $Q$ are conics spanning $120$ tritangent planes $\Gamma_i$ to $C$. (There is also an infinite family of tritangent planes given by planes in $\Pi$ tangent to $Q$. We usually discard the latter and only refer to the former as the tritangent planes). 

Then $H_i = \mathrm{span}(l, \Gamma_i) \subset \bP^5$ is a $\bP^3$ that intersects $X$ in a cubic threefold with six nodes in general position. Hence $H_i$ is a so-called \emph{symmetroid hyperplane}, and each of these contains two families of scrolls $[T_i], [T_i']$ with $[T_i]+[T_i']=2\eta_X$. 

\begin{theorem}\label{tExplicitConstruction}
Let $X$ be a very general cubic fourfold with an involution of type $\phi_2$. Then an equation for $X$ can be written in 1120 ways as 
\[
\det (M) + L_1L_2L_3 =0
\]
where $M$ is a $3\times 3$ matrix of linear forms and $L_i$ are linear forms, such that the corresponding 1120 Gale dual cubics are non-isomorphic derived Fourier--Mukai partners of $X$. Each of these has an involution of type $\phi_1$ and in general exactly one Eckardt point. 
\end{theorem}

\begin{proof}
It suffices to exhibit one example of such an $X$. 

Consider the following eight points $\mathcal{P}=\{P_1, \dots , P_8\}$ in $\bP^2$ defined over $\bZ$ (the Macaulay2 computations below are then done reducing into characteristic $100,000,007$): 
\[
\begin{pmatrix}
1 & 0 & 0 & 1 & 1 & 1 & 1 & 1 \\
0 & 1 & 0 & 1 & 2 & 3 & 4 & 5 \\
0 & 0 & 1 & 1 & 398 & 1704 & 4449 & 4665
\end{pmatrix}.
\]
The blowup $\Sigma = \mathrm{Bl}_{\mathcal{P}}(\bP^2)$ is a degree $1$ del Pezzo surface. 

Let $u_0, u_1$ be a basis for the cubic curves through $\mathcal{P}$, i.e. the anti-canonical system of $\Sigma$. Then $h^0 (-2K_{\Sigma})=4$ and contains $u_0^2, u_0u_1, u_1^2$. Let $v$ be a further sextic with nodes in $\mathcal{P}$ completing this to a basis. The linear system with this basis maps $\Sigma$ to $\bP^3$ with coordinates $x_2, \dots , x_5$ and the image $Q$ is given by 
\[
\det \begin{pmatrix}
    x_2 & x_3 \\
    x_3 & x_4
\end{pmatrix} = 0.
\]
Let $C$ be the branch curve of the double cover $\varphi$ as above. There is a $\bP^4$ of cubics containing $C$. We choose an explicit one $S$ with an equation $s=0$ given by  
\begin{align*}
&x_0^2x_2+5x_2^3+2x_0x_1x_3 
-44352523x_2^2x_3+5x_2x_3^2+5x_3^3+x_1^2x_4-25248805x_2^2x_4\\
&-35620459x_2x_3x_4+5x_3^2x_4+35659795x_2x_4^2-35376841x_3x_4^2 
-27300683x_4^3\\
&-14083726x_2^2x_5+25749314x_2x_3x_5+5x_3^2x_5
+12756811x_2x_4x_5+1156676x_3x_4x_5\\
&+49572192x_4^2x_5
+41713330x_2x_5^2+12858843x_3x_5^2-36773330x_4x_5^2-
32382654x_5^3.   
\end{align*}
Consider the conic bundle over $\bP^3$ branched in $S\cup Q$ given by
\[
N= \begin{pmatrix}
    x_2 & x_3 & 0 \\
    x_3 & x_4 & 0  \\
    0   & 0 & s
\end{pmatrix}
\]
and the corresponding cubic $X$ given by 
\[
F:=(x_0, x_1, 1) N (x_0, x_1, 1)^t = 0. 
\]
The exceptional curves $E_i$ can be explicitly constructed as curves in $\bP^2$ with prescribed singularities in $\mathcal{P}$ \cite{Sturmfels18}. The 240 exceptional curves come in 120 pairs such that the curves in each pair have intersection number $3$. The curves in each such pair have the same image under $\varphi$. We thus obtain 120 tritangent planes as above. The three tangent points are the images of the three intersection points.  In this way we get 120 explicit symmetroid hyperplanes $H_i$ as above. 

The cubic threefolds $Y_i =X\cap H_i$ are 6-nodal and can be written as $\det (T_i)$ where $T_i$ is a $3\times 3$ matrix of linear forms such that $\phi_2 (T_i)=T_i^t$. We find such matrices by the following ansatz. Let
\[
M = 
\begin{pmatrix}
   0 & x_0 & 0 \\
- x_0 & 0 & x_1 \\
0 & -x_1 & 0
\end{pmatrix}
+
\begin{pmatrix}
x_4 & z_0 & x_3 \\
z_0 & z_1 & z_2&  \\
x_3 & z_2 & x_2
\end{pmatrix}
\]
with $z_i = \sum_{j=2}^4 q_{i,j}x_j$ and $q_{i,j}$ unknowns, and solve for 
\[
    \det M = Y_i.
\]
This works in all $1120$ cases.
Notice that the $x_0,x_1$ part is antisymmetric and the $x_2,x_3,x_4$-part is symmetric, so $\phi_2$ will act on $M$ by transposition.

Let $E_i$ and $E_j$ be two exceptional curve with intersection number $2$. Then the scrolls given by the rows and columns of the matrices $T_i$ and $T_j$ form a nonsyzygetic pair in the sense that they lie in different hyperplanes and intersect with intersection number $1$ or $5$. We can then find a matrix $M_{ij}$ such that in the symmetroid hyperplane $H_i$ it reduces to $T_i$ and in the $H_j$ to $T_j$  \cite[Prop. 2.2]{BBM25}. We find $M_{ij}$ by an ansatz. Let
\[
M_{ij} = M_i + x_5M_{sym}
\]
where $M_{sym}$ is a generic symmetric matrix. Then eliminate $x_5$ again by using the equation $h_j$ to obtain $N_j$ which depends on the entries of $M_{sym}$. We then  solve the equation
\[
    \det N_j = \lambda \det M_j.
\]
and use the solution $M_{sym}$ to define $M_{ij}$.
This implies that $F-\det M_{ij}$ is divisible by $h_i$ and $h_j$ and therefore  there is another symmetroid hyperplane $H_k$ such that one can write 
\[
F= \det (M_{ij}) + h_ih_jh_k 
\]
where $h_i$ is an equation of $H_i$. 

The construction of Gale dual cubics  from \cite[Definition 4.1]{BBM25} can be made $\bZ/2$ equivariant in the following way: let $v$ be a vector in $\bC^{12}$ and consider the first nine entries as a $3\times 3$ matrix (using lexicographic order) 
\[
\begin{pmatrix}
t_1 & t_4 & t_7 \\
t_2 & t_5 & t_8 \\
t_3 & t_6 & t_9
\end{pmatrix}
\quad \mapsto \quad 
\begin{pmatrix}
t_1 & t_2 & t_3 & t_4 & t_5 & t_6 & t_7 & t_8 & t_9
\end{pmatrix}
\]
and let $\phi$ act by transposing this matrix. The map 
\[
\bC^{12} \xrightarrow{(M_{ij}, h_i, h_j, h_k)} \bC^{6} 
\]
is then equivariant with $\phi_2$ acting on $\bC^6$. Then $\phi$ has three negative eigenvalues on $\bC^{12}$ corresponding to anti-symmetric matrices. Let now
\[
\bC^6 \xrightarrow{(M'_{ij}, h_i', h_j', h_k')^t} \bC^{12}
\]
be the kernel of the preceding displayed map. Since $\phi_2$ has two negative eigenvalues, the restriction $\phi_1$ of $\phi$ to the kernel has one negative eigenvalue. Hence the Gale dual cubic $X_{ij}'$
\[
F_{ij}' = \det (M_{ij}') - h_i'h_j'h_k' =0 
\]
has an involution of type $\phi_1$ and has an Eckardt point $e_{ij}$. It can be checked that this is the only Eckardt point on each of these cubics: 
We count the Eckardt points by computing the degree of the locus where the polar quadric has rank at most $2$ as in \cite{Coskun15} and \cite[Lemma 3.14]{BoBo25}. 

Intersecting $X_{ij}'$ with its tangent space at $e_{ij}$ and afterwards projecting from $e_{ij}$ gives a cubic surface $E_{ij}$. The isomorphism class of $E_{ij}$ is an invariant of the isomorphism class of $X_{ij}$. We consider the invariant $I_8^2/I_{16}$ of \cite[Prop. 12]{ElsenhansJahnel}. Computing this invariant in characteristic $100,000,007$ for each $E_{ij}$ yields $1120$ different values, each one occurring three times since we get the equation 
\[
F= \det (M_{ij}) + h_ih_jh_k 
\]
also from $M_{ik}$ and $M_{jk}$.

This computation took about 10 hours on a MacBook Pro.
\end{proof}

\begin{proposition}\label{p:DualInvolutionCubics}
Let $X$ be a very general cubic fourfold with an involution of type $\phi_2$ and $Y$ one of the 1120 Fourier--Mukai partners with an involution of type $\phi_1$  constructed above. 

We have
\begin{enumerate}
\item 
$X$ and $Y$ are birational. 
\item 
$X$ and $Y$ are Fourier--Mukai partners. 
\item 
The Kuznetsov components $\mathrm{Ku}(X)$ and $\mathrm{Ku}(Y)$ are \emph{not} $\bZ/2$-equivalent. So Huybrechts' Conjecture, in the equivariant context, holds in this example. 
\item 
$X$ and $Y$ are not $\bZ/2$-birational with their involutions $\phi_1$ and $\phi_2$. 
\item 
The Fano varieties $F(X)$ and $F(Y)$ are \emph{not} $\bZ/2$-birational. 
\end{enumerate}
\end{proposition}

\begin{proof}
Part a) is \cite[Theorem 7.2 c)]{BBM25}.

Part b) is also \cite[Theorem 7.2 c)]{BBM25}. 

Part c) follows because the Addington-Thomas Hodge structures $\widetilde{H}(X)$ and $\widetilde{H}(Y)$ cannot be $\bZ/2$-isomorphic simply because $\phi_2$ is symplectic, but $\phi_1$ is not. So $\phi_2$ acts trivially on $H^{1,3}(X)$, but $\phi_1$ acts nontrivially on $H^{1,3}(Y)$. 

Part d) is clear by examining the corresponding Burnside symbols for $X$ and $Y$. To each $G$ variety of dimension $4$, one associates the symbol:
$$ [X\righttoleftarrow G]:=\sum_H\sum_F(H, K \curvearrowright k(F),\beta_F)\in\Burn_4(G).$$ Here $\Burn_4(G)$ is the equivariant Burnside group, defined by \cite[Section 3.3]{tschinkel2023equivariantbirationalgeometrylinear}, and the symbol records: $H\leq G$ a subgroup stabilizing maximal strata $F\subset X$, induced action of a subgroup $K\leq Z_G(H)/H$ on $k(F)$, and corresponding weights $\beta_F$ of $H$ in the normal bundle of $F$ (see \cite{Kresch2020EquivariantBT,tschinkel2023equivariantbirationalgeometrylinear} for more information). The main point is that the symbol $[X\righttoleftarrow G]$ is a $G$-birational invariant.

We analyse the symbols for both $X$ and $Y$. The symbol for $Y$ has a nonvanishing divisorial part in its Burnside symbol coming from the smooth $\bZ/2$-invariant cubic threefold hyperplane section, whereas $X$ doesn't have that in its Burnside symbol (no divisor is fixed). Hence it follows that $X$ and $Y$ are not $\bZ/2$-birational.

Part e) follows in a very similar way as c), using the fact that $\phi_1$ is anti-symplectic whereas $\phi_2$ is symplectic. We use the same notation to denote the induced action on $F(X)$ and $F(Y).$ Suppose there was a birational map $f\colon F(X) \dashrightarrow F(Y)$ with $\phi_1\circ f = f \circ \phi_2$. If $\sigma_Y$ and $\sigma_X$ are the holomorphic symplectic forms on $F(Y)$ and $F(X)$, then $f^* \sigma_Y = \lambda \,\sigma_X$ for $\lambda \in \bC^*$, but 
\[
(\phi_1\circ f)^* \sigma_Y = - f^* \sigma_Y, \quad (f \circ \phi_2)^* \sigma_Y = f^* \sigma_Y
\]
a contradiction because $f^* \sigma_Y$ is nonzero. Hence no $\bZ/2\bZ$-equivariant birational map between $F(X)$ and $F(Y)$ exists. 
\end{proof}

\section{Example 2: Cubic fourfolds with symplectic automorphisms of order $3$}\label{s:SymplecticOrder3}

Next, we consider cubic fourfolds with a symplectic automorphism of order $3$, following \cite{BGM25}.
We will consider a case of automorphism which best highlights all of the technical difficulties of our count. We prove:

\begin{theorem}
    Let $X$ be a general cubic fourfold with symplectic automorphism of order 3 of type $\phi_3^6$ (see Theorem \ref{t:Z3}. Then $X$ has exactly $624$  Fourier--Mukai partners. Of them, $351$ have primitive algebraic lattice isomorphic to $A(X)_{prim},$ and $273$ have a different lattice $K=K_{mystery}$ described more precisely below in Lemma \ref{l:EnumerationLatticesPhi63}.
\end{theorem}

\subsection{Counting Hodge theoretic Fourier--Mukai Partners}
We recall from \cite[Thm. 2.2]{BGM25}: 

\begin{theorem}\label{t:Z3}
		Let \(X = \{F=0\} \subset \mathbb{P}^5\) be a smooth cubic fourfold with a symplectic automorphism \(\phi\in\Aut(X)\) of prime order \(3\). After a linear change of coordinates that diagonalizes \(\phi\), we have \(\phi(x_0:\ldots :x_5)=(\xi^{\sigma_0} x_0:\ldots :\xi^{\sigma_5} x_5)\) and we denote by \((\sigma_0, \ldots, \sigma_5)\) such an action. If \(d\) denotes the dimension of the family \(F_3^i\) of cubic fourfolds endowed with the automorphism \(\phi_3^i\), then we have the following possibilities:
		\vspace{3pt}
		\begin{itemize}\small{

		\item \(\phi_3^3\): \(p=3\), \(\sigma=(0,0,0,0,1,2)\), \(d=8\),
				\[F=L_3(x_0,\dots,x_3)+x_4^3+x_5^3+x_4x_5M_1(x_0,\dots,x_3),\]
				
				\item \(\phi_3^4\): \(p=3\), \(\sigma=(0,0,0,1,1,1)\), \(d=2\),
				\[F=L_3(x_0,x_1,x_2)+M_3(x_3,x_4,x_5),\]

				\item \(\phi_3^6\): \(p=3\), \(\sigma=(0,0,1,1,2,2)\), \(d=8\),
				\[F=L_3(x_0,x_1)+M_3(x_2,x_3)+N_3(x_4,x_5)+\sum_{i=0,1;j=2,3;k=4,5}a_{i,j,k}x_ix_jx_k,\]
				
				}
		\end{itemize}
		where \(L_i,M_i,N_i\) are homogeneous polynomials of degree \(i\), \(a_{i,j,k} \in\mathbb{C}\).
	\end{theorem}

The most interesting case is the $\phi^6_3$-case. Let $X$ be a general cubic with symplectic order 3 automorphism $\phi^6_3$, and let $\eta_X$ be the square of the hyperplane class.  Then we
	 have:
     \begin{itemize}
         \item $A(X)=\langle\eta_X\rangle\oplus A_{prim}(X)$ and
	\[A_{prim}(X)=\begin{pmatrix} 4 & 1 &-2 & 2 & 2 & 1 &-2 & 2 &-1&  1&  2 &-2\\1 & 4  &0 & 2 &-1 &-1 & 1& -1 &-2&  2& -1&  1\\-2  &0 & 4 &-1& -2 & 0 & 2 & 0 &-1 & 1 & 0 & 0\\ 2 & 2 &-1 & 4 & 0 &-1 & 0 & 0  &0  &2 & 0 &-1\\2 &-1 &-2 & 0 & 4 & 0 &-2 & 2 & 1&  0 & 2& -2\\1 &-1 & 0 &-1 & 0 & 4 &-2 & 2 & 0& -1 & 0 &-1\\-2 & 1&  2 & 0& -2 &-2 & 4 &-2 & 0&  0 &-1 & 2\\ 2 &-1 & 0 & 0 & 2 & 2& -2 & 4 & 0&  1 & 2 &-2\\-1 &-2& -1 & 0&  1 & 0 & 0 & 0 & 4& -1 &-1 & 0\\1 & 2  &1 & 2  &0 &-1 & 0 & 1 &-1&  4 & 1 &-1\\ 2 &-1 & 0 & 0 & 2 & 0& -1 & 2 &-1&  1 & 4 &-2\\-2 & 1&  0 &-1& -2 &-1&  2 &-2 & 0& -1& -2 & 4\end{pmatrix},\]
        \item $D(A_{prim}(X)) \simeq (\bZ/3)^6$;
        \item $T:=T(X)=U(3)^2\oplus A_2^3$ with $D(T)\simeq (\bZ/3)^7$.
     \end{itemize}
     
Note that the quadratic forms on both discriminant groups are nondegenerate. Also note that $A_{prim}(X)$ is the so-called \emph{Coxeter-Todd lattice}. 

\medskip

Note that Assumption \ref{assumption 1} clearly holds in this case, and Assumption \ref{assumption 3} holds as well because, with the notation of Theorem \ref{theorem: unique embedding into AT lattice}, we have that $H=\widetilde{\Lambda}$ has signature $(h_+, h_-)=(20, 4)$, $T$ has signature $(t_+, t_-)=(8,2)$, and thus (a) of Theorem \ref{theorem: unique embedding into AT lattice} holds, for (b) we get $h_++h_- -(t_++t_-)= 14$ and thus (c) is vacuous and (b) is true because $14 \ge l(D(T)_3)+2=9.$ 

\medskip

We first determine the list $\mathscr{K}$ as before.

\begin{lemma}\label{l:EnumerationLatticesPhi63}
    We have $\mathscr{K}:=\{A(X)_{prim}, K_{mystery}\}$
\end{lemma}
\begin{proof}
    First, recall that, by \cite[Theorem 13, Chapter 15, \S 8.2]{ConwaySloane}, for a given dimension $\dim L=n$ and determinant $\det L=3^r$ all even $3$-elementary lattices $L$ form one genus, and there exists an even 3-elementary lattice $L$ if and only if either $r\in \{0,n\}$ and $n\equiv 0 (8)$ or $0<r<n$ and $2r\equiv n (4)$.

    We again use Theorem \ref{theorem: prim embedding} for $S=T$ and $M=\Lambda$, corresponding to the two possibilities of the groups $H_M\cong H_\Lambda$.

    The first corresponds to the genus of $A(X)_{prim}$, corresponding to $H_T\cong H_\Lambda\cong \bZ/3$. In this case, $K$ is rank $12$ and is $3$-elementary with determinant $3^6$, and by Venkov-Scharlau all such lattices form one genus. An Oscar/Magma computation \cite{MMSO_Files} shows that there are 10 lattices in that genus and that all except $A_{prim}(X)$ have vectors $v$ of minimal square $v^2=2$, hence are disallowed by Definition \ref{d:List}.    

    The second possibility corresponds to when $H_T=H_\Lambda=(1)$. Such an embedding exists if there exists an even lattice $K$ with signature $(12,0)$, and $D(K)=(\bZ/3)^8$ equipped with a nondegenerate form $q_K$. Using Oscar/Magma \cite{MMSO_Files}, one can find such a representative, hence such a $K$ exists. Moreover, all such $K$ form a single genus. There are 6 lattices in the genus (enumerated via OSCAR/Magma), and 5 of them have vectors $v$ of minimal square $v^2=2$, and are hence discarded. There is one remaining lattice $K_{mystery},$ defined by Gram matrix:
    \[K_{mystery}=\begin{pmatrix}
  4 &  2 &  0 &  0 &  0 & -2 & -2 & -2 & -2 &  1 &  2 &  1 \\
  2 &  4 &  0 &  0 &  2 & -1 & -2 & -2 & -2 & -1 &  2 &  2 \\
  0 &  0 &  4 &  2 &  0 &  0 &  2 &  0 & -1 &  0 & -1 &  0 \\
  0 &  0 &  2 &  4 &  1 &  0 &  2 &  1 & -1 &  0 &  0 &  0 \\
  0 &  2 &  0 &  1 &  4 &  1 &  0 & -1 &  0 & -2 &  0 &  0 \\
 -2 & -1 &  0 &  0 &  1 &  4 &  2 &  2 &  2 & -2 & -2 & -2 \\
 -2 & -2 &  2 &  2 &  0 &  2 &  4 &  2 &  1 & -1 & -2 & -2 \\
 -2 & -2 &  0 &  1 & -1 &  2 &  2 &  4 &  2 & -1 & -2 & -2 \\
 -2 & -2 & -1 & -1 &  0 &  2 &  1 &  2 &  4 & -1 & -2 & -2 \\
  1 & -1 &  0 &  0 & -2 & -2 & -1 & -1 & -1 &  4 &  1 &  1 \\
  2 &  2 & -1 &  0 &  0 & -2 & -2 & -2 & -2 &  1 &  4 &  2 \\
  1 &  2 &  0 &  0 &  0 & -2 & -2 & -2 & -2 &  1 &  2 &  4 \\
\end{pmatrix}.\]
    \end{proof}
Note that there exists $v\in K_{mystery}$ with $v^2=6$ and $div_K(v)=3.$

\begin{lemma}
    We have that $|HFM(X, A(X)_{prim})|=351$.
\end{lemma}
\begin{proof}
    Let $K:=A(X)_{prim}$. Note first that by Proposition \ref{prop: case 1}, the set $VFM(X,K)$ is isomorphic to the set $HFM(X,K)$.  By Proposition \ref{pCount2}, a lower bound for $|VFM(X,K)|$ is the number of $\overline{G}_{K,T}$-orbits in $\calH_{K,T}.$ We determine this number using Proposition \ref{p:CountCase1}. By Theorem \ref{tOrderOrthogonal}, the order of the orthogonal group of a non-degenerate quadratic form on $\bF_3^7$ is equal to
    \[
        2\cdot 3^9 \cdot (3^2-1)(3^4-1)(3^6-1).
    \]
    The determinant of a matrix representing the discriminant bilinear form in this case is $1$ mod $3$. Therefore the orthogonal group of $\bF_3^6$ is $\operatorname{O}^-(6, \bF_3)$ of order
    \[
    2\cdot 3^6 (3^3+1) (3^2-1)(3^4-1) =26127360 .
    \]
    The quotient of the former by the latter is 702, and dividing by $2$ we get 351 as the number of $\overline{G}_{K, T}$-orbits in $\calH_{K, T}$ in this case.

    We show that the homomorphism $\rho \colon G_{K, T} \to \overline{G}_{K, T}$ is surjective, hence that the number of $\overline{G}_{K, T}$-orbits in $\calH_{K, T}$ determined previously is equal to the number of Hodge-theoretic Fourier--Mukai partners with $K=A_{prim}(X)$ by Proposition \ref{pCount2}.  

    Magma computes that the automorphism group of $A_{prim}(X)$ has order $2^{10} \cdot 3^7 \cdot 5 \cdot 7=78382080$. One can verify via Magma \cite{MMSO_Files} that the homomorphism
    \[
    \mathrm{O}(A_{prim}(X)) \to \mathrm{O}(D(A_{prim}(X)))
    \]
    is surjective in this case (the order of the image is 26127360 as computed independently above).
    
    Thus the hypothesis of the last assertion of  Proposition \ref{pCount2} is satisfied and we are counting $|HFM(K)|$ with the number 351 computed above.
    \end{proof}

\begin{remark}
    Note that \cite[Cor. 1.9.6]{nikulin} gives the surjectivity of $\mathrm{O}(K_p)\to \mathrm{O}(D(K_p))$ for a $p$-adic lattice $K_p$,  but that doesn't yield the desired surjectivity in our case because we do not know if $\mathrm{O}(K) \to \mathrm{O}(K_p)$ is surjective. 
\end{remark}

\begin{proposition}
    We have that $|HFM(X, K_{mystery})|=273$.
\end{proposition}
\begin{proof}
    Let $K:=K_{mystery}$. Note that $K$ satisfies Case 2 of Assumption \ref{a:CountViaOrthGroups}. 
With Magma \cite{MMSO_Files} we established the following:
\begin{enumerate}
   
\item  The mystery lattice $K$ has $1116$ vectors $v$ with $v^2=6$, and $36$ of those, $v_i$, $i=1, …, 36$, have divisibility 3 in $K$ in addition. 

\item  The $36$ $v_i$ form one orbit under the automorphism group $\mathrm{O}(K)$. 

\item  $\mathrm{O}(K)$ has order $67184640$. The homomorphism $\mathrm{O}(K) \to \mathrm{O}(D(K))$ is injective, the order of the image is $67184640$, too. 

\item  The images of the $\frac{v_i}{3}$ in $D(K)\simeq (\bZ/3)^8 =: W$ form an $\mathrm{O}(K)$ invariant hyperplane $H$ in the $\bZ/3$-vector space $W$. 

\item  The image of $\mathrm{O}(K)$ in the automorphism group of $H$ is also of order  $67184640$. In other words, the element of $\mathrm{O}(D(K))$ that leaves $H$  pointwise fixed and multiplies by $-1$ in the line $H^{\perp}\subset W $ is not in the image of $ \mathrm{O}(K)$ in $\mathrm{O}(D(K))$.
\end{enumerate}

To compute the $|HFM(X,K)|$ with mystery lattice $K$ of a general $\phi^6_3$-cubic we need to count the number of $\{\pm 1\} \times \mathrm{O}(K)$ orbits of graphs of injective anti-isometries $f: V:= D(T) \rightarrow D(K)$ where $V\simeq (\bZ/3)^7$, where $\{\pm 1\}$ acts in $D(T), \mathrm{O}(K)$ in $D(K)$, such that $f$ actually does give rise to a cubic with that mystery lattice by Torelli. This is equivalent to the image of $f$ being equal to $H\subset W$ by Lemma \ref{lem:VFM'} - this is exactly counting $|VFM(X,K)'|$.

Note that $\{ -1 \} \times \{ -1\} \subset \{\pm 1\} \times \mathrm{O}(K)$ stabilises each such graph. So we can restrict to counting $\mathrm{O}(K)$ orbits of such graphs, this gives the same number. 

Now the number of anti-isometries $f: V=D(T) \to H \subset W $ is equal to the cardinality of the orthogonal group $\mathrm{O}(H)= \mathrm{O}(\bF_3^7)$, which by Theorem \ref{tOrderOrthogonal} is 
$$2\cdot 3^9\cdot (3^2-2)(3^4-1)(3^6-1)=18341406720.$$

Since $\mathrm{O}(K)$ acts faithfully on $H$ by {\em e)} above, the number of $\mathrm{O}(K)$ orbits of the graphs of such anti-isometries is 

$$18341406720: |\mathrm{O}(K)| = 18341406720: 67184640 = 273. $$

\end{proof}

\begin{remark}\label{r:DifferenceVirtualActual}
Notice that in the above Proposition $|VFM(X, K_{mystery})| \ge \frac{|\mathrm{O}(D(K_{mystery}))|}{2|\mathrm{O}(K_{mystery})|}$, which is much larger than the number of Hodge-theoretic Fourier--Mukai. partners. 
\end{remark}

\subsection{Counting actual Fourier--Mukai partners}

\begin{proposition}
    Let $Y\in HFM(X, A(X)_{prim}).$ Then $Y$ is a Fourier--Mukai partner of $X$.
\end{proposition}
\begin{proof}
This is unproblematic in this case: any Hodge-theoretic Fourier--Mukai partner $Y$ in this case has primitive algebraic lattice isomorphic to $A_{prim}(X)$, hence is a cubic with a symplectic automorphism of prime order $3$ again by the Strong Torelli theorem $Aut(Y)\cong Aut_{Hodge}(H^2(Y,\bZ), \eta_X).$ Since the primitive algebraic lattice is that of $X$, there is an automorphism of $H^2(Y,\bZ)_{prim}$, that can be extended to one of $H^2(Y,\bZ)$ fixing the hyperplane class.
Thus by \cite{ouchi}, both $X$ and its partner $Y$ have associated K3s in the derived sense, hence they are Fourier--Mukai partners by Theorem \ref{t:TwistedK3s}. 

\end{proof}

\begin{proposition}
    Let $Y\in HFM(X, K_{mystery})$. Then $Y\in \calC_{42}$, and $Y\in FM(X).$
\end{proposition}
\begin{proof}
    We have that $Y\in \calC_{42}$: indeed, the vector
    \[
    v= (-1,-1,-1,-1,1,-1,0,0,-1,0,0,0) \in K_{mystery}
    \]
    satisfies $v^2=14$, and the sublattice $\langle h^2 , v\rangle$ is primitive in $H^4 (Y, \bZ)$: since 3 and 14 are coprime there is no nontrivial overlattice. It follows that $Y$ has an associated K3 in the derived sense, and hence they are Fourier--Mukai partners by Theorem \ref{t:TwistedK3s}. 
\end{proof}

\begin{remark}
    Note that all $Y\in FM(X)$ are rational by \cite[Thm. 5.12]{RSTrisecantFlops} and \cite[Cor. 1.3]{BGM25} - hence cubic fourfolds with an order three symplectic automorphism $
    \phi_3^6$ satisfy Huybrechts' conjecture.
    
    In contrast to the involution case treated above, we do not know an explicit geometric construction for those $Y\in FM(X,K_{mystery})$ nor for any cubic with that primitive algebraic lattice. 
\end{remark}

\appendix

\section{$\bZ/2$ rationality}\label{r:Z2RationalityPhi3}
In \cite{Mar23}, it was proved that cubic fourfolds admitting an involution of type $\phi_3$ (in the notation of Theorem \ref{thm: inv}) are rational.
Recall that a variety $X$ with $G\leq \Aut(X)$ is \textit{$G$-linearizable} if it $G$-birational to $\bP(V)$ where $V$ is a linear $G$ representation.
Here, by using a parameter count, we will prove that these cubics  are $\bZ/2$-Pfaffian and thus $\bZ/2$-linearizable.

\begin{theorem}
    Let $X$ be a general cubic fourfold with anti-symplectic involution of type $\phi_3$ (in the notation of Theorem \ref{thm: inv}). Let $\bZ/2:=\langle \phi_3\rangle=\Aut(X)$. Then $X$ is $\bZ/2$-linearizable.
\end{theorem}

\begin{proof}
Indeed, by \cite[Prop.2.12]{Mar23}, the involution in the ambient $\bP^5=\bP(V)$, $\dim V=6$, can be assumed to be such that $V = V^+ \oplus V^-$ with $V^+$, $V^-$ the $\pm 1$-eigenspaces and $\dim V^+ = \dim V^- =3$. We show that a general such cubic $X$ can be obtained from a $\bZ/2$-equivariant inclusion
\[
i\colon V \hookrightarrow \Lambda^2 V^* .
\]
Note that choosing a $\bZ/2$-equivariant codimension one subspace $W \subset V $ (there are many of these, $\bZ/2$ being abelian), then allows one to define a $\bZ/2$-birational map 
\begin{align*}
Q_W\colon X & \dashrightarrow \bP (W)\\
[\psi] & \mapsto \mathrm{ker}\, (\psi) \cap W
\end{align*}
where we view a point $x\in X$ as a skew-symmetric map $\psi \colon V \to V^*$. (It needs to be checked that $Q_W$ is birational for a general choice of $\bZ/2$-invariant $W$, which can be done in one example). 

Note that the datum of an equivariant inclusion $i$ amounts to writing down an invariant $6\times 6$ skew matrix of linear forms in $V^*$. Here we take the format
\[
A=\begin{pmatrix}
+ & + & + & - & - & - \\
+ & + & + & - & - & - \\
+ & + & + & - & - & - \\
- & - & - & + & + & + \\
- & - & - & + & + & + \\
- & - & - & + & + & + \\
\end{pmatrix}
\]
where a $+$ indicates the entry is taken from the $+1$-eigenspace of $V^* \simeq V$, and $-$ entries are taken from the $-1$-eigenspace of $V^*$. There are thus $3+3+9=15$ linear forms as entries above the diagonal in this skew-matrix, which thus depends on $15\cdot 3=45$ parameters. There is a natural action of $\mathrm{GL}(3)\times \mathrm{GL}(3)$ on such matrices $A$ by $A\mapsto BAB^t$. Thus cubics resulting from such a $\bZ/2$-Pfaffian construction depend on $45-18=27$ parameters (counted projectively), and $28$ is precisely the number of $\phi_3$-invariant monomials of degree $3$ in the coordinates on $V^*$. Modding out further by the action by projectivities preserving the decomposition $V= V^+\oplus V^-$, such cubics depend on $27 - (18-1)=10$ moduli, agreeing with the count in \cite[Prop. 2.12 (3)]{Mar23}. Thus a general cubic with a $\phi_3$-involution is $\bZ/2$-rational.
\end{proof}

\bibliographystyle{alpha}
\bibliography{bibliography}
\end{document}